\documentclass[12pt]{amsart}
\usepackage[english]{babel}
\usepackage[utf8]{inputenc}

\usepackage{graphicx}
\usepackage{pinlabel}
\usepackage{amsmath}
\usepackage{amsfonts}
\usepackage{latexsym}
\usepackage{amssymb}
\usepackage{color}

\input xy
\xyoption{all}

\numberwithin{equation}{section}

\newtheorem{thm}{Theorem}[section]
\newtheorem{cor}[thm]{Corollary}
\newtheorem{prop}[thm]{Proposition}
\newtheorem{lem}[thm]{Lemma}
\newtheorem{fact}[thm]{Fact}

\theoremstyle{remark}
\newtheorem{rem}[thm]{Remark}
\theoremstyle{example}

\theoremstyle{definition}
\newtheorem{defn}[thm]{Definition}

\theoremstyle{remark}

\newcommand{\R}{\mathbb{R}}
\newcommand{\F}{\mathbb{F}}

\newcommand{\Z}{\mathbb{Z}}
\newcommand{\C}{\mathbb{C}}

\newcommand{\rank}{\operatorname{rank}}

\newcommand{\Aut}{\mathrm{Aut}}

\newcommand{\id}{\mathrm{id}}

\newcommand{\OP}{\operatorname}

\begin{document}

\title[The stable Morse number as a lower bound]{The stable Morse number as a lower bound for the number of Reeb chords}
\author[G. Dimitroglou Rizell and R. Golovko]{Georgios Dimitroglou Rizell and Roman Golovko}
%\author[Roman Golovko]{Roman Golovko}

\begin{abstract}
Assume that we are given a closed chord-generic Legendrian submanifold $\Lambda \subset P \times \R$ of the contactisation of a Liouville manifold, where $\Lambda$ moreover admits an exact Lagrangian filling $L_{\Lambda} \subset \R \times P \times \R$ inside the symplectisation. Under the further assumptions that this filling is spin and has vanishing Maslov class, we prove that the number of Reeb chords on $\Lambda$ is bounded from below by the stable Morse number of $L_{\Lambda}$. Given a general exact Lagrangian filling $L_{\Lambda}$, we show that the number of Reeb chords is bounded from below by a quantity depending on the homotopy type of $L_{\Lambda}$, following Ono-Pajitnov's implementation in Floer homology of  invariants due to Sharko. This improves previously known bounds in terms of the Betti numbers of either $\Lambda$ or $L_{\Lambda}$.
\end{abstract}

%\address{University of Cambridge, United Kingdom}
%\email{g.dimitroglou@maths.cam.ac.uk}
%\urladdr{http://www.dimitroglou.name/}
%\address{MTA R\'{e}nyi Institute of Mathematics, Budapest, Hungary} \email{golovko.roman@renyi.mta.hu}
%\urladdr{https://sites.google.com/site/ragolovko/}
%\date{\today}
%\thanks{The first author is supported by the grant KAW 2013.0321 from the Knut and Alice Wallenberg Foundation. The second %author is supported by the ERC Advanced Grant LDTBud. A part of this work was done during a visit of the authors to the Institut %Mittag-Leffler (Djursholm, Sweden).}
%\subjclass[2010]{Primary 53D12; Secondary 53D42}

%\keywords{Legendrian, Reeb chords, exact Lagrangian filling, stable Morse number, Arnold's conjecture}

\maketitle

\tableofcontents

\section{Introduction}
\subsection{Motivation}
One of the first striking applications of Gromov's theory of pseudoholomorphic curves \cite{PCISM} was that a closed exact Lagrangian immersion $\widetilde{\Lambda} \subset (P,d\theta)$ inside a Liouville manifold must have a double-point, given the assumption that it is Hamiltonian displaceable. Gromov's result has the following contact-geometric reformulation, which will turn out to be useful. Consider the so-called \emph{contactisation} $(P \times \R,dz+\theta)$ of the Liouville manifold $(P,d\theta)$, which is a contact manifold with the choice of a contact form. Recall that a (generic) exact Lagrangian immersion $\widetilde{\Lambda} \subset (P,d\theta)$ lifts to a Legendrian (embedding) $\Lambda \subset (P \times \R,dz+\theta)$. One says that $\Lambda$ is \emph{horizontally displaceable} given that $\widetilde{\Lambda}$ is Hamiltonian displaceable. The above result thus translates into the fact that a horizontally displaceable Legendrian submanifold $\Lambda$ must have a \emph{Reeb chord} for the above standard contact form --- i.e.~a non-trivial integral curve of $\partial_z$ having endpoints on $\Lambda$. A similar result holds for Legendrian submanifolds of boundaries of subcritical Weinstein manifolds, as proven in \cite{HDATCC} by Mohnke.

In the spirit of Arnold \cite{FSIST}, the following conjectural refinement of the above result was later made: the number of Reeb chords on a chord-generic Legendrian submanifold $\Lambda \subset (P \times \R,dz+\theta)$ whose Lagrangian projection is Hamiltonian displaceable is at least $\frac{1}{2}\sum_i b_i(\Lambda;\F)$. However, as was shown by Sauvaget in \cite{CLED4} by  the explicit counter-examples inside the standard contact vector space $(\R^4 \times \R,dz+\theta_0)$, $\theta_0=-(y_1dx_1+y_2dx_2)$, the above inequality is not true without additional assumptions on the Legendrian submanifold; also, see the more recent examples constructed in \cite{CELIWFDP} by Ekholm-Eliashberg-Murphy-Smith. The latter result is based upon the h-principle proven in \cite{LC} by Eliashberg-Murphy for Lagrangian cobordisms having loose negative ends in the sense of Murphy \cite{LLEIHDCM}.

On the positive side, the above Arnold-type bound has been proven using the Legendrian contact homology of the Legendrian submanifold, under the additional assumption that the Legendrian contact homology algebra is sufficiently well-behaved. Legendrian contact homology is a Legendrian isotopy invariant independently constructed by Chekanov \cite{DAOLL} and Eliashberg-Givental-Hofer \cite{ITSFT}, and later developed by Ekholm-Etnyre-Sullivan \cite{LCHIPR}. This invariant is defined by encoding pseudoholomorphic disc counts in the Legendrian contact homology differential graded algebra (DGA for short) which usually is called the \emph{Chekanov-Eliashberg algebra} of the Legendrian submanifold. In the case when the Chekanov-Eliashberg algebra of a Legendrian admits an augmentation (this should be seen as a form of non-obstructedness for its Floer theory), the above Arnold-type bound was proven by Ekholm-Etnyre-Sullivan in \cite{OILCHAELI} and by Ekholm-Etnyre-Sabloff in \cite{ADESFLCH}. In \cite{ETNORCUALROTCA}, the authors generalised this proof to the case when the Chekanov-Eliashberg algebra admits a finite-dimensional matrix representation, in which case the same lower bound also is satisfied.

The above Arnold-type bound is also related to the one regarding the number of Hamiltonian chords between the zero-section in $T^*L$ (or, more generally, any exact closed Lagrangian submanifold of a Liouville manifold) and its image under a generic Hamiltonian diffeomorphism. Namely, such Hamiltonian chords correspond to Reeb chords on a Legendrian lift of the union of the Lagrangian submanifold and its image under the Hamiltonian diffeomorphism. In fact, as shown by Laudenbach-Sikorav in \cite{PERSISTANCE}, the number of such chords is bounded from below by the stable Morse number of the zero-section (and hence, in particular, it is bounded from below by half of the Betti numbers of the disjoint union of \emph{two} copies of the zero-section). Arnold originally asked whether this bound can be improved, and if in fact the \emph{Morse number} of the zero-section is a lower bound. However, this question seems to be out of reach of current technology. On the other hand, we note that the stable Morse number is equal to the Morse number in a number of cases; see \cite{OTSMNOACM} as well as Section \ref{sec:gendefs} below for more details.

Finally, we mention the remarkable result by Ekholm-Smith in \cite{ELIWODP}, which shows that the smooth structure itself can predict the existence of more double points than the original bound given in terms of the homology. Namely, a $2k$-dimensional manifold $\Sigma^{2k}$ for $k > 2$ that admits a Legendrian embedding having precisely one transverse Reeb chord in the standard contact space must be \emph{diffeomorphic} to the standard sphere unless $\chi(\Sigma^{2k})=-2$. Also see \cite{ELIWODPR} for similar results in other dimensions.

\subsection{Results}
In this paper, we will explore a priori lower bounds for the number of Reeb chords on a Legendrian submanifold $\Lambda \subset (P \times \R,dz+\theta)$, given that it admits an exact Lagrangian filling $L_\Lambda \subset (\R \times P \times \R,d(e^t(dz+\theta)))$ inside the symplectisation. Recall that the condition of admitting an exact Lagrangian filling is invariant under Legendrian isotopy; see e.g.~\cite{LCOLK}. The bound will be given in terms of the simple homotopy type of $L_\Lambda$. First, we recall that such a Legendrian submanifold automatically has a well-behaved Chekanov-Eliashberg algebra; namely, an exact Lagrangian filling induces an augmentation by \cite{RSFTLLCHALFC}. In the case when the projection of $\Lambda$ to $\widetilde{\Lambda} \subset P$ is displaceable, the aforementioned result can thus be applied, giving the above Arnold-type bound. However, in this case, there are even stronger bounds that can be obtained from the topology of the exact Lagrangian filling $L_\Lambda$ (and without the assumption of horizontal displaceability). See Section \ref{sec:wrapped} below for previous such results as well as an outline of the proof, which is based upon Seidel's isomorphism in wrapped Floer homology. This is also the starting point of the argument that we will use in order to prove our results here.

In the following we assume that a Legendrian submanifold $\Lambda \subset (P \times \R,\alpha:=dz+\theta)$ is chord-generic and has an exact Lagrangian filling $L_\Lambda \subset (\R \times P \times \R,d(e^t\alpha))$. Here $t$ denotes the coordinate on the first $\R$-factor. In particular, the set of Reeb chords $\mathcal{Q}(\Lambda)$ of $\Lambda$ is finite. Further, the set of Reeb chords $c$ in degree $|c|=\OP{CZ}(c)-1 \in \Z / \Z\mu_{L_\Lambda}$ will be denoted by $\mathcal{Q}_{|c|}(\Lambda)$, where the grading is induced by the Conley-Zehnder index modulo the Maslov number $\mu_{L_\Lambda} \in \Z$ of $L_\Lambda$ as defined in \cite{TCHOLSIR}. Observe that $\mu_{L_\Lambda}=0$ in particular implies that the first Chern class of $(P,d\theta)$ vanishes on $H_2(P)$.

For a group $G$ being the epimorphic image of $\pi_1(L_\Lambda)$, consider the Morse homology complex $(CM_{\bullet}(L_\Lambda,f;R[G]),\partial_f)$ of $L_\Lambda$ with coefficients in the group ring $R[G]$ twisted by the fundamental group, where $R$ is a unital commutative ring and $f\colon L_\Lambda \to \R$ is a Morse function satisfying $df(\partial_t)>0$ outside of a compact set. (The generators of this complex are graded by the Morse index, and the differential counts negative gradient flow lines.)
\begin{thm}\label{mainthmsmplehomeqnts}
Let $L_\Lambda \subset (\R \times P \times \R,d(e^t\alpha))$ be an exact Lagrangian filling of an $n$-dimensional closed Legendrian submanifold $\Lambda \subset ( P \times \R,\alpha)$ with fundamental group $\pi:=\pi_1(L_\Lambda)$ and Maslov number $\mu_{L_\Lambda} \in \Z$.
\begin{itemize}
\item[(i)] In the case when the filling is spin and when $\mu_{L_\Lambda}=0$, the Morse homology complex $(CM_\bullet(L_\Lambda,f;\Z[\pi]),\partial_f)$ is simple homotopy equivalent to a $\Z[\pi]$-equivariant complex $(\Z[\pi]\langle \mathcal{Q}_{n-\bullet}(\Lambda) \rangle,\partial)$;
\item[(ii)] In the general case, it follows that the complex $$(CM_\bullet(L_\Lambda,f;R[G]),\partial_f)$$ is homotopy equivalent in the category of $G$-equivariant complexes to a complex $(R[G]\langle \mathcal{Q}_{n-\bullet}(\Lambda) \rangle,\partial)$ with grading in $\Z/\mu_{L_\Lambda}\Z$. Here we can always take $R=\Z_2$, while we are free to choose an arbitrary unital commutative ring in the case when $L_\Lambda$ is spin.
\end{itemize}
\end{thm}
We prove Theorem~\ref{mainthmsmplehomeqnts} in Section~\ref{maintheoremanditsconsequences}. Now let $\OP{stabMorse}(M)$ denote the stable Morse number of a manifold $M$ with possibly non-empty boundary, see Definition \ref{defn:stablemorse}. Using Theorem~\ref{mainthmsmplehomeqnts} and the adaptation of \cite[Theorem 2.2]{OTSMNOACM} to the case of manifolds with boundary (see Proposition \ref{prop:morse}), the following result is immediate:
\begin{cor}\label{maininequalitystableMorsenumberofafilling}
Suppose that $\Lambda \subset  P \times \R$ is a chord-generic closed Legendrian submanifold admitting an exact Lagrangian filling $L_\Lambda$ which is spin and has vanishing Maslov number. It follows that the bound
\begin{align}\label{mainineqstMrc}
|\mathcal{Q}(\Lambda)| \ge \OP{stableMorse}(L_\Lambda)
\end{align}
is satisfied for the number of Reeb chords on $\Lambda$.
\end{cor}
By using the long exact sequence in singular homology of the pair $(\overline{L_\Lambda},\partial{\overline{L_\Lambda}}=\Lambda)$, where $\overline{L_\Lambda}$ denotes the compact part of $L_\Lambda$, we obtain the following inequalities
\[ \OP{stableMorse}(L_\Lambda) \ge \sum_i b_i(L_\Lambda;\F) \ge \frac{1}{2}\sum_i b_i(\Lambda;\F),\]
for any field $\F$. Obviously, Inequality~\eqref{mainineqstMrc} is a strengthening of the original Arnold-type bound. For a discussion about how to construct examples of exact Lagrangian fillings for which our obtained lower bound is strictly greater than previously known bounds in terms of the homology of the filling, we refer to Section \ref{sec:examples}.

Note that in forthcoming work \cite{UCLCHIP}, Eriksson \"{O}stman also obtains an improved version of the above Arnold-type bound for certain horizontally displaceable Legendrian submanifolds. The bound is obtained in terms of the homotopy type of the Legendrian submanifold itself. It does not assume the existence of an exact Lagrangian filling, but rather assumes the existence of an augmentation of a version of the Chekanov-Eliashberg algebra having twisted coefficients that is defined in the same article.

In the course of showing the above result, we also obtain the following generalisation of the aforementioned result by Sikorav-Laudenbach \cite{PERSISTANCE}, which also is related to the theory of stable intersection numbers as introduced by Eliashberg-Gromov in \cite[Section 2.3]{LITFDA}.
\begin{thm}
\label{thm:stablefloer}
Consider a closed exact Lagrangian submanifold $L \subset (P,d\theta)$ which is spin and has vanishing Maslov number. For any $k \ge 0$, the exact Lagrangian submanifold $L \times \R^k \subset (P \times \C^k,d\theta \oplus \omega_0)$ with $\omega_0=dx_1\wedge dy_1+\dots +dx_k\wedge dy_k$ satisfies the property that
\[ \# (L \times \R^k) \pitchfork \phi^1_{H_s}(L \times \R^k) \ge \OP{stableMorse}(L),\]
given that the above intersection is transverse, and that the Hamiltonian is of the form $H_s = f_s+Q$, where:
\begin{itemize}
\item $Q(x_1\!+\!iy_1,\hdots,x_k\!+\!iy_k)\!=\!Q(x_1,\hdots,x_k)$ is a non-degenerate quadratic form on $\R^k \subset \C^k$; and
\item $f_s \colon P \times \C^k \to \R$, $s\in [0,1]$, satisfies that $\max_{s \in [0,1]}\|f_s\|_{C^1}$ is bounded for a product Riemannian metric of the form $g_P \oplus g_{\OP{std}}$ on $P \times \C^k$. Here we moreover require $g_P$ to be invariant under the Liouville flow on $(P,d\theta)$ outside of a compact subset, while $g_{\OP{std}}$ denotes the Euclidean metric.
\end{itemize}
\end{thm}
\begin{rem}
Damian's examples in \cite{OTSMNOACM} (see Theorem \ref{thm:damian}) can be used to produce a Hamiltonian for which $(L \times \R^k) \pitchfork \phi^1_{H_s}(L \times \R^k)$ is strictly less than the Morse number of $L$, given that $k \gg 0$ is sufficiently large.
\end{rem}

We also get the following two theorems which are consequences of Theorem~\ref{mainthmsmplehomeqnts} together with the algebraic machinery developed by Ono and Pajitnov in \cite{OTFPOAHDIPOFG}. For a finitely presented group $G$, we denote by $d(G) \in \Z_{\ge 0}$ the minimal number of generators of $G$.

\begin{thm}\label{intrineqabssi}
Let $\mu_{L_\Lambda}=0$. Assume that $\pi_1(L_{\Lambda})$ admits a finite epimorphic image $G$, which is a simple or solvable group.
\begin{itemize}
\item[(i)]
Under the above assumptions, we have
$$|\mathcal Q(\Lambda)|\geq d(G) + \sum_{i\neq 1} {b_i(L_{\Lambda};\F)};$$
\item[(ii)]
If moreover $\pi_1(L_{\Lambda})$ is a finite perfect group, then $$|\mathcal Q(\Lambda)|\geq d(G) + \sum_{i\neq 1} b_i(L_{\Lambda};\F) + 2.$$
\end{itemize}
Here we have to use the field $\F=\Z_2$ unless $L_\Lambda$ is spin, in which case it can be chosen arbitrarily.
\end{thm}

\begin{thm}\label{intrineqrelsimaslnonz}
Assume that $\pi_1(L_{\Lambda})$ admits a finite epimorphic image $G$, which is a simple or solvable group.
\begin{itemize}
\item[(i)]
Under the above assumptions, we have
$$|\mathcal Q(\Lambda)|\geq \max (1,d(G)-1) + \sum_{i\neq 1} b_i(L_{\Lambda};\F),$$
where $i\in \mathbb Z/\mu_{L_\Lambda} \mathbb Z$;
\item[(ii)]
If moreover $\mu_{L_\Lambda}\geq 2n+2$, then $$|\mathcal Q(\Lambda)|\geq d(G) + \sum_{i\neq 1} b_i(L_{\Lambda};\F),$$
where $i\in \mathbb Z/\mu_{L_\Lambda} \mathbb Z$.
\end{itemize}
Here we have to use the field $\F=\Z_2$ unless $L_\Lambda$ is spin, in which case it can be chosen arbitrarily.
\end{thm}
Note that the estimates presented in Theorems~\ref{intrineqabssi} and \ref{intrineqrelsimaslnonz} are in general weaker than the estimate described in Corollary~\ref{maininequalitystableMorsenumberofafilling}. On the other hand, the estimates from Theorems~\ref{intrineqabssi} and \ref{intrineqrelsimaslnonz} hold in the less restrictive settings compared to the settings of Corollary~\ref{maininequalitystableMorsenumberofafilling}. Theorems~\ref{intrineqabssi} and \ref{intrineqrelsimaslnonz} will be proven in Section~\ref{estimPajOnogensection}.

In Section~\ref{constrproductanddouble}, we provide a construction of exact Lagrangian fillings with a given finitely presented fundamental group. This  leads to examples where the estimate described in the second part of Theorem~\ref{intrineqabssi} coincides with the stable Morse number of an exact Lagrangian filling and, moreover, such that this bound is better than the estimate coming from the homological data of the filling. Finally in Section \ref{solvableexamplesfarfromseidelsiso}, we provide a series of examples of exact Lagrangian fillings for which the estimates for the number of Reeb chords provided by Theorems~\ref{intrineqabssi} and \ref{intrineqrelsimaslnonz} are arbitrary far from the estimates coming from the so-called Seidel's isomorphism in Theorem \ref{thm:seidel} (i.e.~coming from homological data of the filling).

\subsection{Previous results obtained using wrapped Floer homology}
\label{sec:wrapped}

It was previously known that a Legendrian submanifold $\Lambda \subset (P \times \R,dz+\theta)$ admitting an exact Lagrangian filling $L_\Lambda$ in the symplectisation satisfies a stronger form of the above Arnold-type bound. It should also be noted that, in this case, the bound is in fact true also without the assumption of horizontal displaceability of $\Lambda$. Namely, as outlined in \cite[Conjecture 1.2]{RSFTLLCHALFC} and later developed in \cite{LPPTTSOPTRAA} and \cite{OHRAFOELC} by the first author and by both authors, respectively, the number of Reeb chords for such a chord-generic Legendrian submanifold $\Lambda$ is at least $\sum_i b_i(L_\Lambda;\F)$. (Recall that there is an inequality $\sum_i b_i(L_\Lambda;\F) \ge \frac{1}{2} \sum_i b_i(\Lambda;\F)$.) In the case of exact Lagrangian fillings inside a more general subcritical Weinstein domain, this result was proven by Ritter in \cite[Theorem 11.1]{TQFTSOSC}.

These results are all proven using roughly the same idea, based upon computations of the wrapped Floer homology of the filling which in these cases is acyclic. Wrapped Floer homology, originally defined in \cite{AOSAOVF} by\linebreak Abouzaid-Seidel and in \cite{RSFTLLCHALFC} in a different form by Ekholm, generalises Floer's original Lagrangian intersection Floer homology \cite{MorseTheoryLagr} to the setting of exact Lagrangian fillings. Note that wrapped Floer homology is  always acyclic in our setting, since the exact Lagrangian fillings considered here are displaceable in the appropriate sense.

Since the argument in the proof of the above bound is the starting point of the method that we will be using here, we now give a brief outline:

First, the wrapped Floer homology computed for the pair $(L_\Lambda,L_\Lambda')$, where $L_\Lambda'$ is any Hamiltonian push-off of $L_\Lambda$, is acyclic. For a suitable push-off $L_\Lambda'$, the wrapped Floer complex can thus be made into an acyclic mapping cone of a chain map from the Morse homology complex of $L_\Lambda$, as follows from Floer's original computation, to a subcomplex whose underlying vector space is given by $\F\langle \mathcal{Q}_{n-\bullet}(\Lambda)\rangle$. This acyclic mapping cone gives rise to the so-called Seidel's isomorphism:
\begin{thm}[Seidel]
\label{thm:seidel}
Let $\Lambda$ be a Legendrian submanifold of $P\times \R$ with the property that $\Lambda$ admits an exact Lagrangian filling $L_\Lambda$. Then there is a quasi-isomorphism
$$H^\bullet(\overline{L_{\Lambda}},\partial \overline{L_\Lambda};\F) \xrightarrow{\simeq} HF^{+\infty}_\bullet(L_\Lambda;\F),$$
where the right-hand side is the homology of a complex with underlying vector space $\F \langle \mathcal{Q}_{\bullet-1}(\Lambda) \rangle$. In the case when $\OP{char}\F \neq 2$, we must assume that $L_{\Lambda}$ is spin and choose an appropriate spin structure.
\end{thm}
We refer to \cite{OHRAFOELC} for a proof in the setting under consideration here, based upon the ideas in \cite{RSFTLLCHALFC} and \cite{LPPTTSOPTRAA}.

In particular, the above isomorphism implies that the number of Reeb chords on $\Lambda$ is at least $\sum_i b_i(L_\Lambda;\F)$, given that $\Lambda$ is chord-generic. Our main result Theorem~\ref{mainthmsmplehomeqnts} can be interpreted as an upgrade of Seidel's isomorphism to a \emph{simple homotopy equivalence}.

Related ideas were also present in \cite[Theorem 4.7]{FHALC}, where the authors together with Chantraine and Ghiggini used wrapped Floer homology with twisted coefficients in order to show that a Legendrian submanifold in $P \times \R$ with a single Reeb chord satisfies the property that any of its exact Lagrangian fillings must be contractible.

\subsection*{Acknowledgements}
We would like to thank Fran\c{c}ois Charette and Jarek Kedra for very helpful conversations and interest in our work.
In addition, we are grateful to the referee of an earlier version of this paper for many valuable comments and suggestions.
The first author is supported by the grant KAW 2013.0321 from the Knut and Alice Wallenberg Foundation. The second author is supported by the ERC Advanced Grant LDTBud. A part of this work was done during a visit of the authors to the Institut Mittag-Leffler (Djursholm, Sweden).

\section{Preliminaries}
\subsection{Basics from symplectic and contact topology}
\label{sec:gendefs}
By a {\em Liouville manifold} we mean a pair $(P,\theta)$ consisting of an even-\linebreak dimensional smooth manifold $P$ and a one-form $\theta \in \Omega^1(P)$ for which $d\theta$ is a symplectic form, i.e.~for which $(d\theta)^{\wedge \dim P/2}$ is a volume form on $P$. For us, a Liouville manifold will moreover always have a cylindrical convex (sometimes called positive) non-compact end. In other words, we will assume that $(P,d\theta)$ is of the form $((0,+\infty) \times Y, d(e^s\alpha_Y))$ in the complement of a compact sub-domain with smooth boundary. Here $s$ is the standard coordinate on the $(0,+\infty)$-factor and $\alpha_Y \in \Omega^1(Y)$ is a one-form on $Y$. (The latter exact symplectic manifold is the half of the the symplectisation of the closed contact manifold $(Y,\alpha_Y)$.) Recall that a Liouville manifold possesses the so-called \emph{Liouville vector field} $X \in \Gamma TP$ which is $d\theta$-dual to $\theta$, i.e. satisfying the equation $i_{X}d\theta=\theta$.

Given an exact symplectic $2n$-manifold $(P,d\theta)$, we define its {\em contactisation} to be $(P\times \R, dz + \theta)$, where $z$ is a coordinate on the $\R$-factor. It is not difficult to see that $\alpha:=dz + \theta$ satisfies $\alpha\wedge (d\alpha)^{\wedge n}\neq 0$, and hence $\alpha$ is a {\em contact form} on $P\times \R$ which defines a {\em contact structure} $\xi:=\ker \alpha \subset T(P \times \R)$.

An $n$-dimensional submanifold $\Lambda\subset P\times \R$ is called {\em Legendrian} given that $T\Lambda \subset \xi$, and a smooth $1$-parameter family of Legendrian submanifolds is called a {\em Legendrian isotopy}.

The {\em Reeb vector field} $R_{\alpha}$ on $P \times \R$ is uniquely determined by the equations $i_{R_{\alpha}}\alpha=1$, $i_{R_{\alpha}}d\alpha=0$, and is in this case given by $R_\alpha=\partial_z$. A non-trivial integral curve of $R_\alpha$ having endpoints on a Legendrian submanifold $\Lambda$ is called a \emph{Reeb chord} on $\Lambda$, the set of which will be denoted by $\mathcal{Q}(\Lambda)$. We define the \emph{length} of a Reeb chord $c \in \mathcal{Q}(\Lambda)$ to be $\ell(c):=\int_c dz>0$. In this case, Reeb chords are obviously in bijective correspondence with the double-points of the image of $\Lambda$ under the canonical projection to $P$. We call a Legendrian submanifold \emph{chord-generic}  given that this projection is a generic immersion. In particular, a closed chord-generic Legendrian submanifold has a finite number of Reeb chords in the current setting. Observe that any Legendrian submanifold can be made chord-generic after an arbitrarily $C^\infty$-small perturbation through Legendrian submanifolds.

The {\em symplectisation} of the contactisation $(P\times \R, \alpha)$ is the exact symplectic manifold $(\R\times P\times \R, d(e^t\alpha))$, where $t$ denotes the standard coordinate on the first $\R$-factor. An \emph{exact Lagrangian filling} inside the symplectisation is a central object of this article. It is a special case of the following more general concept.
\begin{defn}
Given two Legendrian submanifolds $\Lambda_{-},\Lambda_{+}\subset (P\times \R, \alpha)$, we call a proper embedding $L \subset (\R\times P\times \R, d(e^t\alpha))$ an \emph{exact Lagrangian cobordism from $\Lambda_{-}$ to $\Lambda_{+}$} given that there exists a number $T> 0$ for which:
\begin{itemize}
\item $L\cap ((-\infty,-T] \times P\times \R) = (-\infty,-T] \times \Lambda_{-}$;
\item $L\cap ([T, \infty) \times P\times \R) = [T, \infty) \times \Lambda_{+}$;
\item $L\cap ([-T, T] \times P\times \R)=:\overline{L}$ is a compact manifold with boundary\linebreak $\partial\overline{L}=\Lambda_-\sqcup\Lambda_+$; and
\item The pull-back of $e^t(dz+\theta)$ to $L$ is exact and, moreover, admits a primitive which is globally constant on $L \cap \{ t \le -T \}$ as well as on $L \cap \{ t \ge T \}$.
\end{itemize}
$\Lambda_{+}$ is called the  {\em positive end of $L$}, and $\Lambda_{-}$ is called the {\em negative end of $L$}. If $\Lambda_-=\emptyset$, we call $L$ an \emph{exact Lagrangian filling} of $\Lambda_+$.
\end{defn}
Observe that the property of admitting an exact Lagrangian filling is a Legendrian isotopy invariant, as shown in \cite{LCOLK}. Indeed, a Legendrian isotopy from $\Lambda$ to $\Lambda'$ induces an exact Lagrangian cobordism from $\Lambda$ to $\Lambda'$. Furthermore, two exact Lagrangian cobordisms $L_a, L_b \subset \R \times P \times \R$ from $\Lambda^-_a$ to $\Lambda$ and from $\Lambda$ to $\Lambda^+_b$, respectively, can be concatenated to form an exact Lagrangian cobordism $L_a \odot L_b \subset \R \times P \times \R$ from $\Lambda^-_a$ to $\Lambda^+_b$.

\subsection{Floer homology with twisted coefficients}
The Floer homology of an exact Lagrangian manifold $L$ with itself can be defined using coefficients twisted by the fundamental group as first described by Sullivan in \cite{KTIFFH}, and later by Damian in \cite{FHOTUC}  as well as Abouzaid in \cite{NLWVMCAHQ}. This construction is analogous to the definition of Morse homology with twisted coefficients.

We will rely on the formulation of Floer homology carried out in \cite{KTIFFH}. Due to the non-compact setting, we have to consider only compatible almost complex structures $J$ on $\R \times P \times \R$ which are of a particular form outside of a compact subset, see \cite{OHRAFOELC}. We start by fixing a compatible almost complex structure $J_P$ on $(P,d\theta)$ for which the standard coordinate $e^s$ on the non-compact cylindrical end $((0,+\infty) \times Y,d(e^s\alpha_Y)) \subset (P,d\theta)$ is $J_P$-convex, i.e.~so that $-d(de^s\circ J(\cdot))$ is a symplectic form tamed by $J_P$. The so-called \emph{cylindrical lift $\widetilde{J}_P$ of $J_P$} is the compatible almost complex structure on $(\R \times P \times \R,d(e^t\alpha))$ defined uniquely by the following properties:
\begin{itemize}
\item The canonical projection $\R \times P \times \R \to P$ is $(\widetilde{J}_P,J_P)$-holomorphic; and
\item The almost complex structure $\widetilde{J}_P$ is cylindrical, i.e.~it is invariant under translations of the $t$-coordinate, and satisfies $\widetilde{J}_P(\partial_t)=\partial_z$, as well as $\widetilde{J}_P(\xi)=\xi$.
\end{itemize}
The compatible almost complex structures $J$ on the symplectisation $\R \times P \times \R$ that we will be considering here will all be taken to coincide with a fixed cylindrical lift $\widetilde{J}_P$ outside of a compact subset.

In this situation, the SFT compactness theorem \cite{CRISFT}, together with the monotonicity properties for the symplectic area of pseudoholomorphic curves \cite{SPOHCIACM}, imply that Lagrangian intersection Floer homology can be defined as usual, and that invariance is satisfied for compactly supported Hamiltonian isotopies. To that end, we recall the following fact. Let $L_0 \subset \R \times P \times \R$ be an exact Lagrangian filling, let $L^s_1$, $s \in [0,1]$, be a smooth family of exact Lagrangian fillings that is fixed outside of a compact subset, and let $J_s$ be a smooth family of tame almost complex structures on $\R \times P \times \R$ all which coincide with the above almost complex structure $\widetilde{J}_P$ outside of a fixed compact subset. Further, we assume that all intersection points $L_0 \cap L^s_1$ are contained in a compact subset.
\begin{lem}[Lemma 4.1 \cite{OHRAFOELC}]
There exists a fixed compact subset $K \subset \R \times P \times \R$ that contains any $J_s$-holomorphic Floer strip having a compact image inside $\R \times P \times \R$, boundary on $L_0 \cup L_1^s$, and both punctures mapping to intersection points $L_0 \cap L_1^s$.
\end{lem}

The Floer complex can be defined in many different ways, but we will use the approach taken in \cite{KTIFFH}.  In addition, see \cite{OHRAFOELC}, \cite{FHALC} for the set-up of Floer homology in the same setting as considered here. For a pair of transversely intersecting exact Lagrangian fillings $L_0,L_1 \subset \R \times P \times \R$, the underlying module of the Floer complex will be given by
\[ CF_\bullet(L_0,L_1;R[\pi_1(L_0)]):=R[\pi_1(L_0)]\langle L_0 \cap L_1\rangle,\]
for a unital commutative ring $R$, where each generator has a well-defined grading modulo the Maslov number of $L_0 \cup L_1$, given that we have fixed a choice of Maslov potential; see Section \ref{sec:grading} below for more details.

The differential
\[ \partial \colon CF_\bullet(L_0,L_1;R[\pi_1(L_0)]) \to CF_{\bullet-1}(L_0,L_1;R[\pi_1(L_0)] )\]
is defined roughly as follows. The coefficient $\langle \partial(y),x\rangle$ is the signed count of rigid $(i,J)$-holomorphic Floer strips of the form
\[ u \colon (\{ s+it \in \C; \: s\in \R,\ t \in [0,1]\},\{t=0\},\{t=1\}) \to (\R \times P \times \R,L_0,L_1)\]
having boundary on $L_0 \cup L_1$ and asymptotics to intersection points $x$ and $y$ as $s \to -\infty$ and $+\infty$, respectively. Each such strip moreover contributes with the coefficient in $\pi_1(L_0) \subset R[\pi_1(L_0)]$ obtained by completing the path $u(s+i0)$ to a loop via the concatenation with fixed capping paths in $L_0$ that connect each intersection point in $L_0 \cap L_1$ with the base point. Recall that we must take $R=\Z_2$ in the above count unless both $L_0,L_1$ are spin. In the latter case, the signed count moreover depends on the choice of a spin structure on $L_0 \cup L_1$.

\subsubsection{The grading convention}
\label{sec:grading}
Assume that $$L_0, L_1 \subset (\R \times P \times \R,d(e^t\alpha))$$ are \emph{connected} $(n+1)$-dimensional exact Lagrangian fillings of $\Lambda_0, \Lambda_1 \subset (P \times \R,\alpha)$. We here restrict attention to the special case when $\Lambda_1$ and $\Lambda_0$ differ by a translation in the $\R$-coordinate (i.e.~by the Reeb flow). In this case, the following uniquely defined grading convention modulo the Maslov number of $L_0 \cup L_1$ will be used.

Consider the Legendrian lift of the disconnected exact Lagrangian immersion $L_0 \cup L_1$ to the contactisation $(\R \times P \times \R) \times \R$ of the symplectisation. We choose a lift where the component $L_1$ has been translated sufficiently far in the  \emph{negative} Reeb direction (of the latter contactisation $(\R \times P \times \R) \times \R$) so that all Reeb chords start on $L_1$. Recall that there is a bijective correspondence between Reeb chords $c_p$ on this lift and intersection points $p \in L_0 \cap L_1$.

Fix points $x_i \in \Lambda_i$, $i=0,1$, where $x_i$, $i=0,1$, differ by a translation of the $\R$-coordinate. For each $i=0,1$, we also fix choices of oriented paths $\gamma^{c_p}_i \subset L_i$ from the end-points of the Reeb chord $c_p$ on (the Legendrian lift of) $L_i$ to the point $(T,x_i) \in (\R \times \Lambda_i) \cap L_i$ for some $T \gg 0$ sufficiently large. For an intersection point $p \in L_0 \cap L_1 \subset CF_\bullet(L_0,L_1;R[\pi_1(L_0)])$ we then prescribe the grading
\[|p|:=n+1-\OP{CZ}(c_p)\]
for the Conley-Zehnder index as defined in \cite{NILSIRTNPO}, where the choice of (disconnected) capping path consisting of the path $\gamma^{c_p}_0$ followed by the path $-\gamma^{c_p}_1$ has been used.

We end by noting that, using the above grading conventions, the differential $\partial$ is of index $-1$.

\subsection{The stable Morse number}
We here briefly discuss the notion of the Morse and stable Morse number. We refer to \cite{OTSMNOACM} for more details concerning the closed case.

\begin{defn}
\label{defn:morse}
\hspace{-1ex}For a compact manifold $M$ possibly with non-empty boundary, we call
$$\OP{Morse}(M):=\min\limits_{f \in C^\infty(M,\R)}\{\# \OP{Crit}(f)\ | \: f \: \mbox{Morse}, \:\: df^{-1}(0) \cap \partial M = \emptyset \}$$
the {\em Morse number of $M$}.
\end{defn}

\begin{defn}
\label{defn:almostquadratic}
Let $M$ be a compact manifold possibly with non-empty boundary. A function $F:M\times \R^k\to \R$ is called {\em almost quadratic at infinity} given that there is a non-degenerate quadratic form $Q$ on $\R^k$, and a Riemannian metric $g_M$ on $M$, satisfying the properties that
\begin{itemize}
\item The norm $\|dF-dQ\|_{C^0}$ is bounded for the $C^0$-norm on $T^{\ast}(M\times \R^k)$ induced by the product Riemannian metric $g_M \oplus g_{\OP{std}}$ on $M\times \R^k$, where $g_{\OP{std}}$ denotes the Euclidean metric on $\R^k$; and
\item For the standard coordinate $t \colon (1/2,1] \times \partial M \times \R^k \to (1/2,1]$ on a collar neighbourhood $U \supset \partial M$, we have $dF(\partial_t) > 0$ on $U \times \R^k$.
\end{itemize}
\end{defn}
\begin{defn}
\label{defn:stablemorse}
For a compact manifold $M$ possibly with non-empty boundary, we call
\begin{align*}
&\quad\ \OP{stableMorse}(M)\\
&:=\min\limits_{k \ge 0 \atop F \in C^\infty(M \times \R^k,\R)}\{\# \OP{Crit}(F)\ | \: F \: \mbox{Morse and almost quadratic at infinity}\}
\end{align*}
the {\em stable Morse number of $M$}.
\end{defn}
For any compact smooth manifold $M$ with or without boundary, it immediately follows that
$$\OP{Morse}(M) \geq \OP{stableMorse}(M).$$

The Morse number and stable Morse number coincide for surfaces, the three-sphere \cite{TEFFTRFAIGA}, as well as for simply connected closed manifolds of dimension $k\geq 6$ \cite{OTSOM}.
There are examples due to Damian \cite{OTSMNOACM} of closed manifolds $M$ for which $\OP{Morse}(M)> \OP{stableMorse}(M)$. Note that, if one removes a small open ball from the examples of Damian, the proof of \cite[Theorem~1.2]{OTSMNOACM} (based upon Proposition \ref{prop:morse} below) generalises to show that:
\begin{thm}[Damian]
\label{thm:damian}
Consider the group $G:=\Pi_{i=1}^k A_m$, where $A_m$ is the alternating group on $m$ letters, $m\geq 5$, and $k \ge 0$ is sufficiently large. Then there exists a compact manifold $M$ satisfying $\pi_1(M)= G$ as well as $\OP{Morse}(M)>\OP{stableMorse}(M)$, where $M$ can be taken either with or without boundary.
\end{thm}
Finally, the following statement holds:
\begin{lem}
\label{lem:stablemorse}
A closed smooth manifold $M$ satisfies
$$\OP{stableMorse}(M \times D^k) = \OP{stableMorse}(M)$$
for any $k \ge 0$.
\end{lem}
\begin{proof}
The inequality ``$\ge$'' is immediate, since a function almost quadratic at infinity on $M \times \R^N$ can be suitably stabilised to give a function almost quadratic at infinity on $(M \times D^k) \times \R^N$.

We continue with the inequality ``$\le$''. By the definition of a function $F \colon (M \times D^k) \times \R^N \to \R$ being almost quadratic at infinity, we can write it as $Q + f$ for a quadratic form $Q$ on $\R^N$ together with a function $f$ with a uniform bound on its differential.

Consider the quadratic form $Q_0(\mathbf{x}):=\|\mathbf{x}\|^2$ on $\R^k \ni \mathbf{x}$ together with a bump function $\rho \colon \R \to [0,1]$ satisfying $\rho'(t) \ge 0$, $\rho(t)=0$ for $t \le 1-2\epsilon$, while $\rho(t)=1$ for $t \ge 1-\epsilon$. Also, consider a diffeomorphism $\phi \colon \R^k \to B^k \subset D^k$ of the form $\mathbf{x} \mapsto (\sigma(\|\mathbf{x}\|)/\|\mathbf{x}\|)\mathbf{x}$, where $\sigma \colon \R_{\ge 0} \to [0,1)$ satisfies $\sigma'(t)>0$ together with $\sigma(t)=t$ for all $0 \le t \le 1-\epsilon$. It follows that the function
\[ \widetilde{F} := f \circ (\id_M,\phi,\id_{\R^N}) + \rho(\|\mathbf{x}\|) \cdot Q_0(\mathbf{x})+Q\]
is almost quadratic at infinity on $M \times \R^{k+N}$. By the assumptions on $F$, given that $\epsilon>0$ is sufficiently small, it can readily be seen that the number of critical points of $F$ and $\widetilde{F}$ agree. The inequality now follows.
\end{proof}

\subsection{Simple homotopy theory}
In the following we let $(C_\bullet,\partial)$ be a free and finitely generated chain complex over a group-ring $\Z[G]$ with a preferred graded basis. We also assume that the grading is taken in the integers $\Z$. Observe that we allow generators in negative degrees. In this setting Whitehead defined the notion of a \emph{simple homotopy equivalence} between such complexes in \cite[Section~5]{SHT}; also see Milnor's survey in \cite{WT}. Roughly speaking, two such complexes are simple homotopy equivalent if and only if they are related by stabilisation by trivial complexes, together with a simple isomorphism, i.e.~an isomorphism for which the Whitehead torsion vanishes.

Floer homology has been shown to preserve the simple homotopy type in certain situations, as first shown by Sullivan in \cite{KTIFFH}. These ideas were later successfully used in work by Su\'{a}rez in \cite{ELCAPI}; also, see the related results in \cite{QRT} by Charette. Observe that the setting considered here is similar to that in \cite{ELCAPI} in the sense that we also consider non-compact Lagrangian submanifolds. On the other hand, our setting is simpler, and we do not need the restrictions on the almost complex structure made there.

The following result due to Sullivan is central to us.
\begin{thm}(\cite{KTIFFH})
\label{thm:invariance}
Let $L_0, L_1 \subset \R \times P \times \R$ be exact Lagrangian fillings which are disjoint outside of a compact subset and intersect transversely. If $L_0$ and $L_0'$ are compactly supported Hamiltonian isotopic, it follows that the Floer complexes $(CF_\bullet(L_0,L_1;R[\pi_1(L_0)]),\partial)$ and $(CF_\bullet(L_0',L_1;R[\pi_1(L_0)]),\partial')$ are homotopy equivalent in the category of $\pi_1(L_0)$-equivariant complexes. In the case when the Maslov numbers vanish and the grading is taken in $\Z$, and both fillings are spin and $R=\Z$, then this is a simple homotopy equivalence.
\end{thm}
\begin{proof}
The proof follows from the analysis used to prove \cite[Theorem 3.1]{KTIFFH}. Strictly speaking, the latter result only states that the Whitehead torsion is well-defined for an acyclic Floer complex. The main point of its proof, however, establishes an invariance proof of Floer homology based upon bifurcation analysis that leads to the sought statement.

In particular, \cite[Corollary 3.14]{KTIFFH} concerns the invariance in the case when a handle-slide occurs, while \cite[Lemma 3.15]{KTIFFH} concerns the case when a birth/death of an intersection point occurs. Both statements combine to show that the simple homotopy type of the complex is preserved under these moves. That these cases suffices follows from the techniques in \cite[Section~3.3]{KTIFFH} (in particular, see Theorem 3.12 therein). Namely, there it is established that, after taking a suitable stabilisation, we can perturb the isotopy to one consisting of a sequence of handle-slides and birth/deaths of the required form.

We end with the following remark. As written, the proof in the aforementioned paper only deals with the case $R=\Z_2$. However taking \cite[Remark~5.5]{KTIFFH} into account, together with the construction of coherent orientations in e.g.~\cite{OILCHAELI}, the general case follows as well.
\end{proof}

In order to deduce Corollary~\ref{maininequalitystableMorsenumberofafilling}, we will need the following result, which was proven in \cite{OTSMNOACM} by Damian in the case of a closed manifold.
\begin{prop}[Theorem 2.2 in \cite{OTSMNOACM}]
\label{prop:morse}
Let $M$ be a compact smooth manifold with boundary, and let $(CM_\bullet(M,f;\Z[\pi_1(M)]),\partial_f)$ be the Morse homology complex with twisted coefficients in $\Z[\pi_1(M)]$ of a Morse function $f \colon M \to \R$ for which the gradient of $f$ points outwards along the boundary. Any complex $(D_\bullet,\partial_D)$ that is simple homotopy equivalent to this Morse complex is itself of the form
\[(D_\bullet,\partial_D) = (CM_\bullet(M \times \R^k,F;\Z[\pi_1(M \times \R^k)]),\partial_F),\]
for a Morse function $F \colon M \times \R^k \to \R$ which is almost quadratic at infinity for some $k \ge 0$.
\end{prop}
\begin{proof}
The case when $M$ has boundary follows by the same proof as the case when $M$ is closed. Roughly speaking, the proof consists of the following steps. First, we stabilise the function $f$ by a non-degenerate quadratic function $Q \colon \R^k \to \R$, thus obtaining a Morse function $f + Q \colon M \times \R^k \to \R$ for some sufficiently large $k > 0$. This function is almost quadratic at infinity by construction. We may then realise the simple homotopy equivalence by a sequence of Morse theoretic handle-slide moves together with birth/death moves applied to this stabilised Morse function. Since these moves all may be realised by compactly supported modifications (a gradient flow line connecting two critical points is disjoint from a fixed neighbourhood of $\partial M \times \R^k$ by assumption), we may assume that the Morse function is kept fixed in a neighbourhood of $\partial M \times \R^k$ throughout the modification. See e.g.~\cite[Lemma~2.3]{OTSMNOACM}. The resulting function produced will hence also be almost quadratic at infinity in the sense of Definition \ref{defn:almostquadratic}.
\end{proof}

\subsection{Group theoretic background}
\label{sec:algebra}
Here we remind the reader of some definitions and facts from group theory that will become useful later.

A group $G$ is called an \emph{extension of a group $Q$ by a group $N$}, if $N$ is a normal subgroup of $G$ and the quotient group $G/N$ is isomorphic to the group $Q$.

A group $G$ is called \emph{solvable} if it admits a subnormal series whose factor groups are all abelian, that is, if there are subgroups
$$\{1\}=G_{0}<G_{1}<\dots<G_{k-1}<G_{k}=G$$ such that $G_{i-1}$ is a normal subgroup of $G_i$ and $G_{i}/G_{i-1}$ is abelian for $i=1,\dots,k$.

Let $\pi$ denote a set of primes, then a {\em Hall $\pi$-subgroup} is a subgroup whose order is a product of primes in $\pi$, and whose index is not divisible by any primes in $\pi$.
In \cite{ANOSG}, Hall proved the following theorem.
\begin{thm}[Hall]\label{conjpihallsubgr}
Given a finite solvable group $G$ and a set of primes $\pi$, then any two Hall $\pi$-subgroups of $G$ are conjugate.
\end{thm}

A group $G$ is said to be \emph{perfect} if it equals its own commutator subgroup $[G,G]$.

A group is said to be \emph{superperfect} when its first two homology groups are trivial, i.e. $$H_1(G, \Z) = H_2(G, \Z) = 0.$$ The property of being superperfect is stronger than a property of being perfect, since perfect can be translated into $H_1(G, \Z) = 0$.

The following fact follows from the classification of finite simple groups.
\begin{fact}
If $G$ is a finite simple group, then
\begin{itemize}
\item[(i)] if $G$ is abelian, then the minimal number of generators of $G$ equals $1$,
\item[(ii)] if $G$ is non-abelian, then the minimal number of generators of $G$\linebreak equals~$2$.
\end{itemize}
\end{fact}
Finally, we recall the following realisation result due to Kervaire \cite{SHSATFG}.
\begin{thm}[\cite{SHSATFG}]\label{kerealorh1h2vansupperf}
Let $G$ be a group satisfying the following conditions:
\begin{itemize}
\item[(i)] $G$ admits a finite presentation,
\item[(ii)] $H_{1}(G,\Z)=0$,
\item[(iii)] $H_{2}(G,\Z)=0$,
\end{itemize}
and let $n$ be an integer greater than $4$. Then there exists an $n$-dimensional smooth homology sphere $M$ such that $\pi_1(M)\simeq G$.
\end{thm}
We also observe that if $G$ is a fundamental group of a smooth homology sphere $M$, then it satisfies conditions $(i)$, $(ii)$, $(iii)$.
Conditions $(i)$ and $(ii)$ are automatic and condition $(iii)$ follows from Hopf's theorem \cite{FUZBG} which says that $H_2(G;\Z)\simeq H_2(M;\Z)/\rho(\pi_2(M))$, where $\rho \colon \pi_2(M) \to H_2(M;\Z)$ is the Hurewicz homomorphism.

Given a finitely presented group $G$,
let $d(G)$ denote the minimal number of generators of $G$, and $\delta(G)$ be the minimal number of generators of the augmentation ideal $I_{G}$ of $G$ as a $\mathbb Z[G]$-module (i.e.~the two-sided ideal of $\Z[G]$ generated by elements of the form $g-e$ with $g\in G$ and $e$ being the group unit in $G$). It is not difficult to see that $d(G)-\delta(G)\geq 0$.

Observe that $d(G)=\delta(G)$ holds for a large class of groups.
\begin{thm}[\cite{IR}]
Let $G$ be a finite group which is either simple or solvable. Then $d(G)=\delta(G)$.
\end{thm}
There are also examples of finitely presented groups, where $d(G)-\delta(G)> 0$, see \cite{TPROADPOFG}.

\section{The proof of Theorem~\ref{mainthmsmplehomeqnts} and its consequences}\label{maintheoremanditsconsequences}
Using the Reeb flow of $(P \times \R,\alpha:=dz+\theta)$ we can displace every filling from itself inside the symplectisation $(\R \times P \times \R,d(e^t\alpha))$ (this symplectic manifold is subcritical). We will exploit this displacement in order to create a Hamiltonian isotopy that does the following. We take two copies of the filling suitably perturbed, so that the Floer complex becomes equal to the Morse homology complex for a Morse function on the filling. The goal is then to create a compactly supported Hamiltonian isotopy after which the intersection points are in bijective correspondence with the Reeb chords on the Legendrian end of the filling. The simple homotopy equivalence in Theorem \ref{mainthmsmplehomeqnts} can now be seen to follow from Theorem \ref{thm:invariance}, i.e.~the bifurcation analysis proof of the invariance of Floer homology as performed in \cite{KTIFFH}.

\subsection{The main geometric construction}
In the following we assume that we are given an exact $(n+1)$-dimensional Lagrangian filling $L_\Lambda \subset \R \times P \times \R$ of a closed Legendrian $n$-dimensional submanifold $\Lambda \subset P \times \R$ in the symplectisation of a contactisation. For simplicity, we moreover assume that $L_\Lambda \cap \{ t \ge -1 \} = [-1,+\infty) \times \Lambda$ is cylindrical. (This can always be achieved after a translation of the symplectisation coordinate.)

Recall that the Hamiltonian flow
\[\phi^s_{e^t} \colon (\R \times P \times \R,d(e^t\alpha)) \to (\R \times P \times \R,d(e^t\alpha)) \]
induced by the autonomous Hamiltonian $e^t$ coincides with the Reeb flow of the contact manifold, which in this case simply is a translation of the $z$-coordinate by $s$.

More generally, for any smooth function $g \colon \R \to \R$, we observe that the flow
\begin{gather*}
\Phi^s_{g \partial_z} \colon (\R \times P \times \R,d(e^t\alpha)) \to (\R \times P \times \R,d(e^t\alpha)),\\
(t,p,z) \mapsto (t,p,z+sg(t))
\end{gather*}
is a Hamiltonian flow (and $\phi^s_{e^t}=\Phi^s_{\partial_z}$).

\subsubsection{Constructing a small push-off}
First, we consider the Hamiltonian push-off
\[L_{\Lambda^\epsilon}:=\phi^\epsilon_{e^t}(L_\Lambda), \:\:\epsilon>0,\]
of $L_\Lambda$, which hence is a translation of the $z$-coordinate by $\epsilon>0$. Observe that we have
$$L_{\Lambda^\epsilon} \cap \{ t \ge -1\}=[-1,+\infty) \times \Lambda^\epsilon,$$
where $\Lambda^\epsilon=\phi^\epsilon_{e^t}(\Lambda)$ is obtained by the time-$\epsilon$ Reeb flow applied to $\Lambda$. In particular, $L_{\Lambda^\epsilon}$ is an exact Lagrangian filling of $\Lambda^\epsilon$. See Figure \ref{fig:pushoff} for the case of a one-dimensional filling of a zero-dimensional Legendrian submanifold.

\begin{figure}[htp]
\vspace{1em}
\centering
\labellist
\pinlabel $z$ at 144 89
\pinlabel $t$ at 275 12
\pinlabel $\color{blue}L_{\Lambda^\epsilon}=\phi^\epsilon_{e^t}(L_\Lambda)$ at 190 67
\pinlabel $L_\Lambda$ at 190 19
\pinlabel $-1$ at 103 0
\endlabellist
\includegraphics{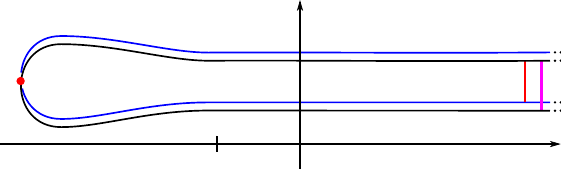}
\caption{The union of $L_\Lambda$ together with its small Hamiltonian push-off $L_{\Lambda^\epsilon}=\phi^\epsilon_{e^t}(L_\Lambda)$. Note that the Reeb chords from the positive end of $L_{\Lambda^\epsilon}$ to the positive end of $L_\Lambda$ are in natural bijective correspondence with the Reeb chords on $\Lambda$.}
\label{fig:pushoff}
\end{figure}
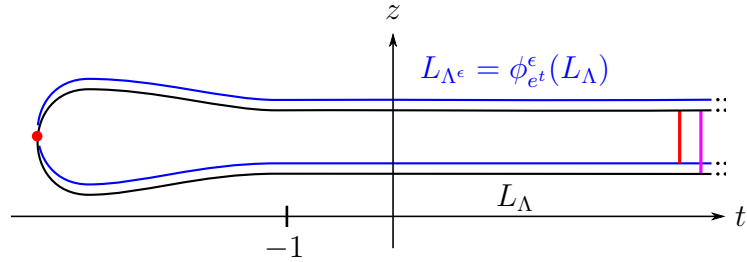

Take a Weinstein neighbourhood of $L_\Lambda$, i.e.~an extension of the Lagrangian embedding $L_\Lambda \hookrightarrow (\R \times P \times \R,d(e^t \alpha))$ to a symplectic embedding of a neighbourhood of the zero-section $L_\Lambda \subset (T^*L_\Lambda,-d(pdq))$. Using this identification, and assuming that we are given an $\epsilon>0$ that is chosen sufficiently small, we may identify $L_{\Lambda^\epsilon} \subset T^*L_\Lambda$ with a section $df$ for a function $f\colon L_\Lambda \to \R$ which satisfies $df(\partial_t)>0$ outside of a compact subset. After a compactly supported Hamiltonian perturbation $L_{\Lambda^\epsilon}'$ of $L_{\Lambda^\epsilon}$ we may assume that the latter function is Morse. The following computation is standard.
\begin{lem}
\label{lem:morsegrading}
Using the grading convention in Section \ref{sec:grading}, it follows that the intersection point $p_c \in L_\Lambda \cap L_{\Lambda^\epsilon}' \subset CF(L_\Lambda,L_{\Lambda^\epsilon}')$ corresponding to $c \in \OP{Crit}(f)$ has grading given by its Morse index, i.e.~$|p_c|=\OP{index}_{f}(c)$.
\end{lem}

\subsubsection{Wrapping}
 We now choose $g$ to be of the form $g(t)=-1$ for $t \le -1$, $g'(t) > 0$ for $t \in (-1,0)$, while $g(t)=0$ for $t \ge 0$. Given that we take $0<\epsilon<\min_{c \in \mathcal{Q}(\Lambda)}\ell(c)$ sufficiently small and $S \gg 0$ sufficiently large, it follows that
\[\Phi^S_{g\partial_z}(L_\Lambda) \cap L_{\Lambda^\epsilon} \subset \{ -1 <t <0 \}.\]
Since $g(t) \le 0$ and $g'(t) > 0$ holds in the subset $\{ -1 <t <0 \}$, the flow $\Phi^s_{g\partial_z}(t,p,z)$ there has the effect of ``wrapping'' the Lagrangian $L_\Lambda$ in the negative $z$-direction. In the same subset $L_\Lambda$ and $L_{\Lambda^\epsilon}$ are cylindrical over $\Lambda$ and $\Lambda^\epsilon$, respectively, and every intersection point thus corresponds to a Reeb chord (i.e.~an integral curve of $\partial_z$) starting on $\Lambda^{\epsilon}$ and ending on $\Lambda$. Note that the latter Reeb chords are in a natural bijective correspondence with the Reeb chords on $\Lambda$. For $S \gg 0$ sufficiently large, we hence get an induced bijection
\begin{gather*}
\mathcal{Q}(\Lambda) \to \Phi^S_{g\partial_z}(L_\Lambda) \cap L_{\Lambda^\epsilon}',\\
c \mapsto p_c
\end{gather*}
between the Reeb chords on $\Lambda$ and the intersection points produced by the wrapping. Figure~\ref{fig:wrap} illustrates this in the case of a one-dimensional filling of a zero-dimensional Legendrian submanifold.

\begin{prop}
\label{prop:disp}
Given $0\!<\!\epsilon\!<\!\min_{c \in \mathcal{Q}(\Lambda)}\ell(c)$ and $S \!\gg\! 0$ sufficiently large, the intersection points $ \Phi^S_{g\partial_z}(L_\Lambda) \cap L_{\Lambda^\epsilon}'$ are in bijective correspondence with $\mathcal{Q}(\Lambda)$, and are all transverse if and only if $\Lambda$ is chord-generic. Moreover, the grading of $p_c \in \Phi^S_{g\partial_z}(L_\Lambda) \cap L_{\Lambda^\epsilon}' \subset CF(\Phi^S_{g\partial_z}(L_\Lambda),L_{\Lambda^\epsilon}')$ is given by $|p_c|=n-|c|$ modulo the Maslov number of $L_\Lambda$, where $n:=\dim \Lambda$.
\end{prop}
\begin{proof}
 The computation of the Conley-Zehnder indices will be performed in a local model of the Legendrian lift $\widetilde{L}_0 \cup \widetilde{L}_1$ of $\Phi^S_{g\partial_z}(L_\Lambda) \cup L_\epsilon$ describing a neighbourhood of the Reeb chord that corresponds to the intersection point $p_c\in \Phi^S_{g\partial_z}(L_\Lambda) \cap L_{\Lambda^\epsilon}$. More precisely, we assume that $\widetilde{L}_0$ and $\widetilde{L}_1$ corresponds to $\Phi^S_{g\partial_z}(L_\Lambda)$ and $L_{\Lambda^\epsilon}$, respectively, while the Legendrian lift is chosen so that the Reeb chord starts at $\widetilde{L}_1$. Recall the definition $|p_c|=n+1-\OP{CZ}(p_c)$ of the grading in Section \ref{sec:grading}, where the latter denotes the Conley-Zehnder index of the Reeb chord for precisely this choice of Legendrian lift.

Let $\Lambda_0,\Lambda_1 \subset (P \times \R,\alpha)$ be small open subsets of the sheets of $\Lambda$ containing the end point and starting point of the Reeb chord $c$, respectively. The sought lifts can now be constructed in the following manner. First, we take the product of $\Lambda_0 \cup \Lambda_1 \subset (P \times \R,dz+\theta)$ with $\R \subset (T^*\R,-pdq)$, and in this way obtain the Legendrian submanifold $\R \times (\Lambda_0 \cup \Lambda_1) \subset (T^*\R \times P \times \R=J^1\R \times P,dz-pdq+\theta)$ of one dimension higher. Second, we deform the component $\R \times \Lambda_1$ by the addition of the one-jet of the function $-Z+\epsilon q$ on $\R$ for $Z \gg 0$ (the number $\epsilon>0$ here corresponds to the translation taking $\Lambda$ to $\Lambda^\epsilon$, while $-Z  \ll 0$ corresponds to a translation of the Legendrian lift of $L_{\Lambda^\epsilon}$ by the negative Reeb flow), giving rise to the Legendrian submanifold $\widetilde{L}_1$. Third, we deform the component $\R \times \Lambda_0$ by the addition of the one-jet of the Morse function $q^2$ on $\R$ (this corresponds to the wrapping of $L_\Lambda$), giving rise to the Legendrian submanifold $\widetilde{L}_0$. The local model produced has a unique Reeb chord $c'$ from $\widetilde{L}_1$ to $\widetilde{L}_0$ contained above $q=\epsilon/2$, the latter being the unique and non-degenerate critical point of the function $f(q)=q^2-(-Z+\epsilon q)$.

The sought identity
\[ |p_c|=n+1-\OP{CZ}(p_c)=n+1-\OP{CZ}(c')=n-|c|\]
now follows from the basic relation
\[\OP{CZ}(c')=\OP{CZ}(c)+\OP{index}_f(\epsilon/2)=\OP{CZ}(c)\]
concerning the Conley-Zehnder index, where $\OP{index}_f(\epsilon/2)=0$ denotes the Morse index. To see this, note that the canonical projection $T^*\R \times P \to P$ maps the Lagrangian projection of $\widetilde{L}_i$ to the Lagrangian projection of $\Lambda_i$ for $i=0,1$. (Above we have used the grading convention in Section \ref{sec:grading} concerning the choices of capping paths.)

We also refer to Figure \ref{fig:wrap} for a schematic picture of the intersection point together with the corresponding Reeb chords.
\end{proof}

\begin{figure}[htp]
\vspace{1em}
\centering
\labellist
\pinlabel $z$ at 144 143
\pinlabel $t$ at 275 65
\pinlabel $\color{blue}L_{\Lambda^\epsilon}=\phi^\epsilon_{e^t}(L_\Lambda)$ at 190 120
\pinlabel $\Phi^S_{g\partial_z}(L_\Lambda)$ at 100 15
\pinlabel $-S$ at 160 4
\pinlabel $-1$ at 103 75
\pinlabel $\color{magenta}c$ at 265 94
\pinlabel $\color{red}\widetilde{c}$ at 245 96
\pinlabel $\color{red}p_c$ at 117 92
\endlabellist
\includegraphics{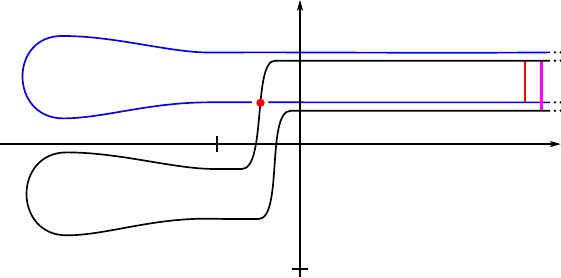}
\caption{After wrapping $L_\Lambda$ by applying the Hamiltonian flow $\Phi^s_{g\partial_z}$, the intersection points $ \Phi^S_{g\partial_z}(L_\Lambda) \cap L_{\Lambda^\epsilon}$ are in natural one-to-one correspondence with the Reeb chords on $\Lambda$. The double point $p_c$ corresponds to the Reeb chord $c$ on $\Lambda$ which, in turn, corresponds to the Reeb chord $\widetilde{c}$ from $\Lambda^\epsilon$ to $\Lambda$.}
\label{fig:wrap}
\end{figure}
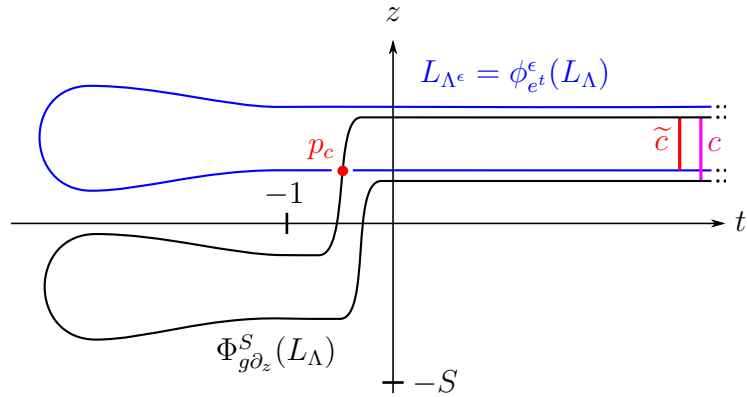

\subsection{Identifying the Floer homology and the Morse homology (the proof of Theorem \ref{mainthmsmplehomeqnts})}

In this subsection, we recall the definition of the Morse homology complex $(CM_\bullet(L_\Lambda,f;R[\pi_1(L_\Lambda)]),\partial_f)$ for a Morse function $f \colon L_\Lambda \to \R$ with coefficients twisted by the fundamental group. The underlying graded module is
\[ CM_\bullet(L_\Lambda,f;R[\pi_1(L_\Lambda)]):=R[\pi_1(L_\Lambda)]\langle \operatorname{Crit}(f)\rangle\]
with grading given by the Morse index. The differential $\partial_f$ counts the number of \emph{negative} rigid gradient flow lines of $f$ defined for a Morse-Smale pair consisting of $f$ together with a Riemannian metric on $L_\Lambda$, while taking the homotopy class of the flow line into account (similarly as to the Floer homology with twisted coefficients as described above).
\begin{prop}
\label{prop:morsecomplex}
Given that $\epsilon>0$ is sufficiently small, after a generic and arbitrarily small compactly supported perturbation $L'_{\Lambda^\epsilon}$ of $L_{\Lambda^\epsilon}$, there is an equality
\[ (CF_\bullet(L_\Lambda,L'_{\Lambda^\epsilon};R[\pi_1(L_\Lambda)]),\partial)=(CM_\bullet(L_\Lambda,f;R[\pi_1(L_\Lambda)]),\partial_f)
\]
of complexes with grading modulo the Maslov number of $L_\Lambda$, given that we use the grading convention specified in Section \ref{sec:grading}. Here $f \colon L_\Lambda \to \R$ is a generic Morse function satisfying the properties that:
\begin{itemize}
\item All critical points are contained in the subset $\{ t < 0 \}$; and
\item Its differential satisfies $df(\partial_t)>0$ for $\{ t \ge 0 \}$.
\end{itemize}
\end{prop}
\begin{proof}
As described in Lemma \ref{lem:morsegrading}, there is a natural identification of the bases of the underlying graded vector spaces.

The calculation of the differential is standard. It was first performed by Floer in \cite{MorseTheoryLagr}, whose computation shows that the Floer homology of a $C^1$-small Hamiltonian push-off of an exact Lagrangian submanifold is equal to a Morse complex. The case with twisted coefficients as considered here was carried out in \cite[Section 2.3]{FHOTUC} by Damian. Also, see the identification in \cite[Proposition 2.4]{ELCAPI}. Recall that, when using the ring $R=\Z$, extra care must be taken when choosing the spin structure in order to obtain the correct signs.
\end{proof}

Using Proposition \ref{prop:morsecomplex} together with Theorem \ref{thm:invariance}, Theorem \ref{mainthmsmplehomeqnts} is now a direct consequence. Together with Proposition \ref{prop:morse} we then conclude Corollary \ref{maininequalitystableMorsenumberofafilling}.

\section{The proof of Theorem \ref{thm:stablefloer}}
This result roughly follows the ideas above, albeit under a slightly different setting. We start by describing the setup.

\subsection{Lagrangian fillings of Legendrian submanifolds inside $P \times S^{2k-1}$}
The exact Lagrangian submanifold that we will be considering here is of the form $L \times \R^k \subset (P \times \C^k,d\theta \oplus d\alpha_0)$, where
\[ \alpha_0 :=\frac{1}{2}\sum_{i=1}^k(x_idy_i-y_idx_i)\]
and $d\alpha_0=\omega_0$ is the standard symplectic form. This Lagrangian submanifold can be considered as an exact Lagrangian filling of a Legendrian submanifold inside $(P \times S^{2k-1},\theta \oplus \alpha_0)$ in the following way, where we recall that $(S^{2k-1},\alpha_0)$ is the standard contact form on the sphere $S^{2k-1} \subset \C^k$.

Choose a Weinstein neighbourhood of $L \subset (P,d\theta)$ which symplectically identifies a neighbourhood of $L$ with a neighbourhood of the zero-section of $(T^*L,d\lambda_L)$, where $\lambda_L=pdq$ denotes the Liouville form. The exactness of $L\subset (P,d\theta)$ implies that $\lambda_L=df+\theta$ holds inside this neighbourhood for some smooth function $f \colon T^*L \to \R$. Using a bump-function $\varphi$ supported in a neighbourhood of $L$ and replacing $\theta$ with form $\theta+d(\varphi f)$, we may thus assume that the primitive of the symplectic form vanishes along $L$ (recall that $L$ is embedded!). In other words, the non-compact Lagrangian submanifold $L \times \R^k \subset (P \times \C^k,d\theta \oplus \omega_0)$ is a cylinder over the Legendrian embedding
\[ L \times (\R^k \cap S^{2k-1}) \subset (P \times S^{2k-1},\theta \oplus \alpha_0)\]
of $L \times S^{k-1}$ outside of a compact subset.

We will call a (possibly time-dependent) Hamiltonian $H_s \colon P \times \C^k \to \R$ \emph{homogeneous at infinity} if it coincides with a function satisfying $H(p,r\mathbf{z})=r^2H(\mathbf{z})$ outside of a compact subset of $P \times \C^k$, where $r \in \R_{\ge 0}$, $p \in P$, and $\mathbf{z} \in \C^k$. Observe that the image of $L \times \R^k$ under the isotopy induced by a homogeneous Hamiltonian is still cylindrical over a Legendrian submanifold outside of a compact subset. To see this, we use the fact that $(t,\mathbf{z}) \mapsto e^{t/2}\mathbf{z}$ is the Liouville flow on $(\C^k \setminus \{0\},d\alpha_0)$. In other words, the latter symplectic manifold can be identified with the symplectisation of $(S^{2k-1},\alpha_0)$, and a homogeneous Hamiltonian induces an isotopy which is the lift of a contact isotopy on the latter contact manifold.

Recall that Floer homology again can be defined for Lagrangian fillings that are cylindrical over Legendrian submanifolds in the above sense, given that we e.g.~choose an almost complex structure which is cylindrical with respect to the convex end of the product Liouville manifold $(P \times \C^k,d\theta \oplus d\alpha_0)$. As usual, invariance of the Floer complex holds for compactly supported Hamiltonian perturbations.

\subsection{Making the Hamiltonian homogeneous at infinity (the proof of Theorem \ref{thm:stablefloer})}
Equip $P \times \C^k$ with the product metric $g_P \oplus g_{\OP{std}}$. By the assumptions of Theorem \ref{thm:stablefloer}, the Hamiltonian $H_s \colon P \times \C^k \to \R$ satisfies $H_s=f_s+Q$, where $f_s \colon P \times \C^k \to \R$, $s \in [0,1]$, has a uniform bound on $\|f_s\|_{C^1}$, and where $Q : \R^k \to \R$ is a non-degenerate quadratic form. In particular, the Hamiltonian vector field associated to $H_s$ is of the form
\[X_s=Y_s+i\nabla Q,\]
where $Y_s \in T(P \times \C^k)$ is uniformly bounded and where $\nabla Q$ denotes the gradient of $Q$ with respect to the Euclidean metric.
\begin{lem}
\label{lem:first}
For any Hamiltonian $H_s \colon P \times \C^k \to \R$ as in the assumption, the intersections
\[(L \times \R^k) \cap \phi^s_{H_s}(L \times \R^k) \subset P \times \C^k\]
for any $s \in [0,1]$ are all contained inside a fixed compact subset $K \subset P \times \C^k$, where this compact subset moreover may be taken to only depend on the norm $\max_{s \in [0,1]}\| df_s \|_{C^0}$.
\end{lem}
\begin{proof}
Fix a constant $C>0$. Given that $R \gg 0$ is sufficiently large, the Hamiltonian vector field $Y_s+i\nabla Q$ may be supposed to satisfy $\|i\nabla Q\| \ge C$ in the complement of
$$P \times \{ \|\mathfrak{Re}(\mathbf{z})\| \le R \} \subset P \times \C^k.$$
In particular, the term $i\nabla Q$ may be assumed to be considerably larger than the Hamiltonian vector field induced by $f_s$ in the same complement. Since the image of
$$ (L \times \R^k)\cap (P \times \{ \| \mathfrak{Re}(\mathbf{z})\|\le R \})= L \times \{ \|\mathfrak{Re}(\mathbf{z})\| \le R \}$$
is compact, its image under $\phi^s_{H_s}$, $s \in [0,1]$, can be assumed to be contained inside a compact subset $K$ as in the assumption. The statement now follows.
\end{proof}

\begin{lem}
\label{lem:second}
After deforming the Hamiltonian $H_s \colon P \times \C^k \to \R$ outside of a compact subset, we may obtain a Hamiltonian $G_s \colon P \times \C^k \to \R$ which is homogeneous at infinity and for which
\[ (L \times \R^k) \cap \phi^1_{G_s}(L \times \R^k) = (L \times \R^k) \cap \phi^1_{H_s}(L \times \R^k)\]
is satisfied. We can moreover take $G_s=g_s+Q$ for $g_s \colon P \times \C^k \to \R$ compactly supported and $Q$ equal to the above non-degenerate quadratic form.
\end{lem}
\begin{proof}
The sought Hamiltonian will be taken to be of the form $G_s:=\chi\cdot f_s+Q$, with the corresponding Hamiltonian vector field $\widetilde{Y}_s+i\nabla Q$, for a smooth cut-off function $\chi \colon P \times \C^k \to [0,1]$ having compact support. It follows that this Hamiltonian is homogeneous at infinity. Observe that the vector field $\widetilde{Y}_s$ has a uniform $C^0$-bound expressed in terms of $\|d\chi\|_{C^0}$, $\|f_s\|_{C^0}$, and $\|df_s\|_{C^0}$.

The required behaviour concerning the intersections can be achieved in the following way. Take a smooth cut-off function satisfying $\chi \equiv 1$ in a sufficiently large subset, while satisfying the uniform bound $\|d\chi\|_{C^0} \le 1$.

In particular, we require that $\chi \equiv 1$ holds in the compact subset
\[\phi^{[0,1]}_{H_s}((\phi^1_{H_s})^{-1}(K)) \subset P \times \C^k,\]
foliated by Hamiltonian trajectories, where $K$ denotes the compact subset produced by Lemma \ref{lem:first}. This is done in order to ensure that the latter Hamiltonian trajectories all are unaffected by the cut-off function $\chi$.

After choosing the compact subset $K$ even larger, we may further assume that $\phi^s_{\chi\cdot f_s+Q}((L \times \R^k) \setminus K)$ is contained in a subset where $\|i \nabla Q\| \ge C$ holds, for an arbitrary fixed constant $C>0$. (In particular, the term $i \nabla Q$ can again be assumed to be considerably larger than the Hamiltonian vector field induced by either $f_s$ or $\chi \cdot f_s$ in the complement of $K$.) The sought property of the intersection points follows from this.
\end{proof}

\begin{lem}
\label{lem:third}
Let $G_s \colon P \times \C^k \to \R$ be a Hamiltonian of the form $g_s + Q$, where $g_s \colon P \times \C^k \to \R$ is compactly supported and $Q$ is a non-degenerate quadratic form on $\R^k$. For each $\epsilon>0$ sufficiently small, one can construct a Hamiltonian $\widetilde{G}_s \colon P \times \C^k \to \R$, where
\begin{enumerate}
\item $\widetilde{G}_s$ coincides with $\epsilon G_{\epsilon s}$ outside of $P \times B^{2k}_R$ for some $R \gg 0$ sufficiently large;
\item The intersection points satisfy
\[ (L \times \R^k) \cap \phi^1_{\widetilde{G}_s}(L \times \R^k)=(L \times \R^k) \cap \phi^1_{G_s}(L \times \R^k);\]
\item The two Lagrangian submanifolds $\phi^1_{\widetilde{G}_s}(L \times \R^k)$ and $\phi^\epsilon_{G_s+h_s}(L \times \R^k)$ are compactly supported Hamiltonian isotopic for any smooth and compactly supported Hamiltonian $h_s \colon P \times \C^k \to \R$.
\end{enumerate}
\end{lem}
\begin{proof}
For any small $\epsilon>0$ we choose a suitable smooth function $\rho_\epsilon \colon \R_{\ge 0} \to \R_{\ge 0}$ satisfying $\rho_\epsilon'(t) > 0$, $\rho_\epsilon(t)=t$ for all $t \in [0,A]$, while $\rho_\epsilon(t)=\sqrt{\epsilon}t$ for all $t\ge B$, and where $B>A>0$ have been chosen sufficiently large. Using this function we then construct the Hamiltonian
\[ \widetilde{G}_s := g_s+Q((\rho_\epsilon(\|\mathbf{x}\|)/\|\mathbf{x}\|)\mathbf{x}).\]
Observe that, since $\mathbf{x} \mapsto (\rho_\epsilon(\|\mathbf{x}\|)/\|\mathbf{x}\|)\mathbf{x}$ is a diffeomorphism of $\R^k$ fixing the origin, the critical points of $Q((\rho_\epsilon(\|\mathbf{x}\|)/\|\mathbf{x}\|)\mathbf{x})$ correspond bijectively to the critical points of $Q$. More precisely, the unique critical point is still the origin $\mathbf{x}=0$.

(1): Given that $A \gg 0$ is chosen sufficiently large, a suitable deformation of $\widetilde{G}_s$ supported outside of some big compact subset of $P \times \{\|\mathfrak{Re}(\mathbf{z})\| \le A\}$ yields the desired Hamiltonian, which we again denote by $\widetilde{G}_s$.

(2): Again given that $A \gg 0$ was chosen sufficiently large, we may assume the following. In the subset of $P \times \C^k$, where $G_s$ and $\widetilde{G}_s$ differ, these Hamiltonians are of the form $Q(\mathbf{x})$ and $Q((\rho_\epsilon(\|\mathbf{x}\|)/\|\mathbf{x}\|)\mathbf{x})$, respectively. Moreover, we may assume that neither function has a critical point in this subset. The property now follows.

(3): Consider the image of $L \times \R^k$ under the one-parameter family
$$\lambda \mapsto \phi^{\epsilon}_{G_s+h_s-\lambda(h_s+G_s-(1/\epsilon)\widetilde{G}_{s/\epsilon})}, \:\: \lambda \in [0,1],$$
of Hamiltonian diffeomorphisms, where $h_s+G_s-(1/\epsilon)\widetilde{G}_{s/\epsilon}$ is compactly supported by part (1). Recall that
$$\phi^{\epsilon s}_{(1/\epsilon)F_{s/\epsilon}}(L \times \R^k)=\phi^s_{F_s}(L \times \R^k), \:\: s \in [0,1],$$
for any Hamiltonian $F_s$.
\end{proof}

\begin{lem}
\label{lem:forth}
For any given Morse function $f \colon L \to \R$, there exists a suitable compactly supported cut-off function $\chi \colon \R^k \to [0,1]$ for which there is an equivalence
\begin{align*}
&\quad\ (CM_{\bullet}(L,f;\Z[\pi_1(L)]),\partial_f) \\
&= (CM_{\bullet+\OP{index}(Q)}(L \times \R^k,\chi\cdot f+Q;\Z[\pi_1(L \times \R^k)]),\partial)
\end{align*}
of complexes.
\end{lem}
\begin{proof}
This can be seen by using the following standard technique. The function $Q \colon L \times \R^k \to \R$ has a non-degenerate critical manifold $L \times \{0\}$ in the Bott sense. We proceed to construct a suitable Morse function of the form $\chi\cdot f +Q$ that gives rise to the sought complex.

First we compute
\[ \nabla(\chi \cdot f+Q) =\chi\nabla f+X+\nabla Q \]
where $\OP{supp}X \!\subset\! \OP{supp}d\chi$ and also $\|X\|_{C^0} \!\le\! \|d\chi\|_{C^0} \|f\|_{C^0}$ are satisfied. Choosing the cut-off function $\chi$ to satisfy $\|d\chi\|_{C^0} \le 1$ together with $\OP{supp}d\chi \subset L \times (\R^k \setminus B^k_R)$ for some sufficiently large $R \gg 0$ the statement follows. Namely, in this case, all critical points of $\chi \cdot f+Q$ are contained inside $L \times \{0\} \subset L \times \R^k$ and correspond to critical points of $f$. Moreover, we can assume that a gradient flow line which connects two critical points cannot leave $L \times \{0\}$ for the product metric.
\end{proof}

We are now ready to prove Theorem \ref{thm:stablefloer}. First, by Lemma \ref{lem:second} we may replace $H_s$ by a Hamiltonian $G_s=g_s+Q$ being homogeneous at infinity, where $g_s$ has compact support and where $Q$ is a nondegenerate quadratic form. Choose a Morse function $f \colon L \to R$ and take a cutoff function $\chi$ as provided by Lemma \ref{lem:forth}. A standard computation in Floer homology (see Proposition \ref{prop:morsecomplex}) together with Lemma \ref{lem:forth} now gives an equality
\begin{align*}
&\quad\ (CF_\bullet(L\times \R^k,\phi^\epsilon_{h+Q}(L \times \R^k);\Z[\pi_1(L_\Lambda)]),\partial)\\&=(CM_{\bullet-\OP{index}(Q)}(L,f;\Z[\pi_1(L)]),\partial_f)
\end{align*}
of complexes for a suitable extension $h$ of $\chi \cdot f$ to all of $P \times \C^k \supset L \times \R^k$, and where $\epsilon>0$ must be chosen sufficiently small. Applying Theorem \ref{thm:invariance} to the compactly supported Hamiltonian isotopy produced by part (3) of Lemma \ref{lem:third} we obtain a simple homotopy equivalence from the latter Morse homology complex to a Floer homology complex generated by the intersection points  $(L \times \R^k) \cap \phi^1_{\widetilde{G}_s}(L \times \R^k)$, where $\widetilde{G}_s$ is produced by the same lemma applied to $G_s$. Finally, using part (2) of Lemma \ref{lem:third}, we may assume that
$$(L \times \R^k) \cap \phi^1_{\widetilde{G}_s}(L \times \R^k)=(L \times \R^k) \cap \phi^1_{G_s}(L \times \R^k),$$
where we recall that the latter intersection points are equal to $(L \times \R^k) \cap \phi^1_{H_s}(L \times \R^k)$ by the construction of $G_s$. It is now simply a matter of applying Proposition \ref{prop:morse} in order to obtain the sought inequality.

\section{Adaptation of the estimates of Ono-Pajitnov}\label{estimPajOnogensection}

In this section, we describe  lower bounds for the number of Reeb chords on $\Lambda$ in terms of $d(G)$ and $\delta(G)$, where $G$ is a group which is an epimorphic image of $\pi_1(L_{\Lambda})$.

The estimates here are all obtained by direct applications of the results from \cite{OTFPOAHDIPOFG} by Ono-Pajitnov concerning the number of generators of $\pi_1$-equivariant complexes enforced by the complexity of the fundamental group $\pi_1$. These results are related to invariants due to Sharko \cite{FOMAATA}. The algebro-topological results from \cite{OTFPOAHDIPOFG} applies to complexes that are $\pi_1$-equivariantly homotopy equivalent to a complex induced by a $\pi_1$-equivariant CW-complex being the universal cover of a \emph{connected} space having fundamental group equal to $\pi_1$ (after possibly reducing the grading $\Z \to \Z / \mu \Z$). In particular, they can be applied to our setting.

In fact, the latter article also provided important inspiration to our work here. The original application of the algebro-topological results therein was to establish lower bounds for the number of periodic orbits of time-dependent Hamiltonian vector field on a closed and weakly monotone symplectic manifold under the assumption that all periodic orbits are non-degenerate.

In the following we will use the notation $c_i:=|\mathcal{Q}_i(\Lambda)|$ for the number of Reeb chords in degree $i$ modulo the Maslov number of $L_\Lambda$. We also write $n:=\dim \Lambda$.

\subsection{Proof of Theorem~\ref{intrineqabssi}}
Let $L_\Lambda$ be an exact Lagrangian filling of a Legendrian $\Lambda\subset P\times \R$ such that $L_{\Lambda}$ is spin and $\mu_{L_{\Lambda}}=0$. In this case, the first part of Theorem~\ref{mainthmsmplehomeqnts} holds. Now we adapt the results of Ono-Pajitnov from \cite[Section 5.1]{OTFPOAHDIPOFG} to this settings.

Using \cite[Theorem 3.10]{OTFPOAHDIPOFG} and \cite[Corollary 3.11]{OTFPOAHDIPOFG}, we get the following proposition.
\begin{prop}\label{firstineqfordanddeltafinim}
The following estimates hold:
\begin{itemize}
\item If $|\pi_1(L_{\Lambda})|>1$, then $c_{n-1}(\Lambda)\geq 1$.
\item If $\pi_1(L_{\Lambda})$ admits a group epimorphism $\pi_1(L_{\Lambda})\to G$ with $|G|>1$, then
\begin{itemize}
\item[(i)] $c_{n-1}(\Lambda)\geq \delta(G)$;
\item[(ii)] $c_{n-1}(\Lambda)\geq d(G)$ if $G$ is simple or solvable;
\item[(iii)] $c_{n-1}(\Lambda)\geq 2$ if $G$ is not cyclic.
\end{itemize}
\end{itemize}
\end{prop}
From  \cite[Corollary 3.20]{OTFPOAHDIPOFG} and \cite[Corollary 3.28]{OTFPOAHDIPOFG} we get the following bound.
\begin{prop}\label{secineqfornmin2usdirref}
If $|\pi_1(L_{\Lambda})|<\infty$, then
\begin{itemize}
\item[(i)]  $c_{n-2}(\Lambda)\geq \delta(\pi_1(L_{\Lambda})) - \dim H_1(L_{\Lambda};\F) + \dim H_2(L_{\Lambda};\F)$ for every field $\F$;
\item[(ii)] $c_{n-2}(\Lambda)\geq  \dim H_2(L_{\Lambda};\F) + 2$ if $\pi_1(L_{\Lambda})$ is perfect;
\end{itemize}
\end{prop}

Propositions~\ref{firstineqfordanddeltafinim} and \ref{secineqfornmin2usdirref} lead to Theorem~\ref{intrineqabssi}.

\subsection{Proof of Theorem~\ref{intrineqrelsimaslnonz}}
In the case when we do not make the assumption that $\mu_{L_{\Lambda}}=0$ and that $L_{\Lambda}$ is spin, then the second part of Theorem~\ref{mainthmsmplehomeqnts} holds.
Then using algebraic machinery described in \cite[Section 4]{OTFPOAHDIPOFG}, we get the following estimates, which are reformulations of the estimates from \cite[Section 5.2]{OTFPOAHDIPOFG}.

\begin{prop}\label{fthgabwithoutassumption}
The following estimates hold:
\begin{itemize}
\item If $|\pi_1(L_{\Lambda})|>1$, then $c_{n-1}(\Lambda)\geq 1$.
\item If $\pi_1(L_{\Lambda})$ admits an epimorphism onto a finite group $G$, then
\begin{itemize}
\item[(i)] $c_{n-1}(\Lambda)\geq \max (1,\delta(G)-1)$, and $|\mathcal Q(\Lambda)|\geq \delta(G)$;
\item[(ii)] if $G$ is simple or solvable, then $|\mathcal Q(\Lambda)|\geq d(G)$;
\item[(iii)] if $G$ is not cyclic, then $|\mathcal Q(\Lambda)|\geq 2$.
\end{itemize}
\end{itemize}
\end{prop}

\begin{prop}\label{sthsdjjkdfjwithoutas}
If $\mu_{L_{\Lambda}}\geq 2n+2$ and $\pi_1(L_{\Lambda})$ admits an epimorphism
onto a finite group $G$, then we have
\begin{itemize}
\item[(i)] $c_{n-1}(\Lambda) \geq \delta(G)$;
\item[(ii)] if $G$ solvable or simple, then $c_{n-1}(\Lambda) \geq d(G)$;
\item[(iii)] if $G$ is not cyclic, then $c_{n-1}(\Lambda) \geq 2$.
\end{itemize}
\end{prop}

Using Propositions~\ref{fthgabwithoutassumption} and \ref{sthsdjjkdfjwithoutas}, we get Theorem~\ref{intrineqrelsimaslnonz}.

\section{Examples of exact Lagrangian fillings}
\label{constrproductanddouble}

Here we describe a general construction of Lagrangian fillings in the symplectisation of the standard contact vector space $(\R^{2n+1},dz+\theta_0)$, where $\theta_0:=-(y_1dx_1+\hdots+y_kdx_k)$. The goal is to construct examples of fillings diffeomorphic to $N \times \R^{k+1}$, where $N$ is a closed manifold, to which our results can be applied in order to produce non-trivial lower bounds for the number of Reeb chords on the Legendrian end $N \times S^k \subset \R^{2n+1}$.

\subsection{Fillings in the symplectisation of the standard contact space}

\label{sec:examples}

Consider the standard disc filling $L^{k+1}_0 \subset (\R \times J^1\R^k,d(e^t\alpha_0))$ of the standard Legendrian $k$-sphere
\[\Lambda^{k}_0 \subset (J^1\R^k=\R^{2k+1},\alpha_0:=dz+\theta_0)\]
of $\OP{tb}(\Lambda^k_0)=(-1)^{k(k-1)/2+1}$. In particular, the Maslov class of $L^{k+1}_0$ vanishes for topological reasons. The filling $L^{k+1}_0$ can either be constructed by hand, or by observing that the contact manifold $J^1\R^k=\R^{2k+1}$ endowed with the standard contact structure is contactomorphic to the complement of a point $(S^{2k+1} \setminus \{\OP{pt}\},\xi_0)$ of the standard tight contact sphere. Namely, under this identification, $\Lambda^k_0$ is identified with $\R^{k+1} \cap S^{2k+1} \subset \C^{k+1} \cap S^{2k+1}$ while $L^{k+1}_0$ can be identified with a compactly supported perturbation of $\R^{k+1} \subset \C^{k+1}$ that misses the origin.

First, we observe that
\[ N \times L^{k+1}_0 \subset (T^*N \times (\R \times J^1\R^k),d(e^t(\lambda_N+\alpha_0))) \]
is an exact Lagrangian filling of $N \times \Lambda^k_0 \subset (J^1(N \times \R^k),dz+\lambda_N+\theta_0)$ diffeomorphic to $N \times S^k$, where $\lambda_N$ denotes the Liouville one-form on $T^*N$. Observe that the Maslov class of $N \times L^{k+1}_0$ also vanishes.

Second, given that $N$ embeds into $\R^{\dim N+k}$ with trivial normal bundle, it follows that there exists a contact-form preserving embedding
\[J^1(N \times \R^k)=T^*(N \times \R^k) \times \R \hookrightarrow T^*\R^{\dim N+k}\times \R=J^1\R^{\dim N+k}.\]
For example, this embedding can be obtained starting from the canonical open exact symplectic embedding $T^*(N \times \R^k) \hookrightarrow T^*\R^{\dim N+k}$ induced by a choice of open embedding $N \times \R^k \hookrightarrow \R^{\dim N+k}$, and then taking the canonical lift to the corresponding jet spaces.

\begin{rem}
We recall that the standard representative of $\Lambda^k_0 \subset \R^{2k+1}$ has a single transverse Reeb chord in degree $k$. Using this fact, we see that the Legendrian embedding $N \times \Lambda^k_0 \hookrightarrow \R^{\dim N+k}$ produced above has a Bott manifold of Reeb chords that, moreover, is diffeomorphic to $N$. A generic perturbation of this Legendrian embedding can be seen to produce precisely a number $\OP{Morse}(N)$ of transverse Reeb chords.\end{rem}

\subsection{Constructing exact Lagrangian fillings with a given fundamental group}
For any finitely presented group $G$, the above method can be used to construct an exact Lagrangian filling with fundamental group equal to $G$.
\begin{prop}\label{constrrealizoffundgroup}
For a given finitely-presented group $G$, there exists a closed $n$-dimensional manifold $M_G$ such that $\pi_1(M_G)= G$ and a Legendrian submanifold $M_G\times \Lambda^k_0\subset (\R^{2(n+k)+1},dz+\theta_0)$ which admits an exact Lagrangian filling diffeomorphic to $M_G\times \R^{k+1}$ for any $k \gg 1$ sufficiently large, where this Lagrangian filling moreover is spin and has vanishing Maslov number.
\end{prop}

\begin{proof}
Recall that the strong Whitney embedding theorem implies that a stably parallelisable manifold $M$ embeds into $\R^{2\dim M}$ with a stably trivial normal bundle. In particular, $M$ embeds into $\R^{2\dim M+1}$ with a trivial normal bundle. (See \cite[Lemma 3.3]{GOHS}.) The construction in Section \ref{sec:examples} together with the following standard result thus proves the claim.
\end{proof}

\begin{lem}\label{realizfundgrparclman}
Given a finitely presented group $G$, there exists a  stably parallelisable closed manifold $M_G$ for which $\pi_1(M_G)=G$.
\end{lem}
\begin{proof}
Observe that there exists a cell complex $X$ with $\pi_1(X)=G$, see \cite[Proposition 1.26, Corollary 1.28]{AT}. Then we embed $X$ into $\R^n$ for some $n$ and thicken it to the open manifold $X^{\OP{op}}$ so that $X^{\OP{op}}$ is a compact manifold of codimension $0$ in $\R^n$ which is stably parallelisable. Then, we define $M_G$ to be a double of $X^{\OP{op}}$ in $\R^{n+1}$, i.e.~we smoothen the corners of the boundary of $X^{\OP{op}}\times [0,1]\subset \R^{n+1}$, thus producing a closed submanifold of $\R^{n+1}$. Since $X^{\OP{op}}$ is stably parallelisable, $X^{\OP{op}}\times [0,1]$ is stably parallelisable as well. Since $M_G$ is the boundary of a stably parallelisable manifold, it is hence itself stably parallelisable.
\end{proof}

\subsection{Exact Lagrangian fillings diffeomorphic to stabilised homology spheres}
In this section, we discuss (integral) homology spheres, as well as related exact Lagrangian cobordisms having interesting fundamental groups. Recall that any (integral) homology sphere is stably parallelisable by \cite{SEATOM}. (The proof of this statement is a modification of the proof of stably parallelisability of homotopy spheres, see \cite{GOHS}.) Using this fact together with Proposition~\ref{constrrealizoffundgroup} we conclude the following. For a given homology sphere $S$, there exists a Legendrian $S\times \Lambda_0^k$ admitting an exact Lagrangian filling diffeomorphic to $S\times \R^{k+1}$ with fundamental group $\pi_1(S\times \R^{k+1})=\pi_1(S)$, given that $k \gg 1$ is chosen sufficiently large.
\subsubsection{The Poincar\'{e} homology sphere}\label{Poinchomsphest}

The Poincar\'{e} homology $3$-sphere (or the dodecahedral space of Poincar\'{e}) is a classical space that has received particular attention. Namely, it was the first example of a homology sphere which is not a sphere, and it also lies in a class of three manifolds closely related to Platonic solids.

The Poincar\'{e} homology sphere can be described as a Brieskorn $3$-sphere. More precisely, consider a polynomial $f: \C^3\to \C$ given by $f(z_1,z_2,z_3)=z_1^2+z_2^3+z_3^5$ and a singular complex variety $f^{-1}(0)$. Note that $f^{-1}(0)$ is singular only at the origin, i.e. when $z_i=0$ for $i=1,2,3$. The Poincar\'{e} homology $3$-sphere $S^3_P$ is defined to be $S^5\cap f^{-1}(0)$, where $S^5\subset \C^3$ is a small $5$-sphere around the origin. There are many other interesting descriptions of the Poincar\'{e} homology sphere, see \cite{EFOTPHS}.

Note that $\pi_1(S^3_P)$ is isomorphic to the so-called binary icosahedral group $\langle 2,3,5 \rangle$ defined by the presentation
\begin{align*}
\langle s,t\ |\ (st)^2=s^3=t^5 \rangle.
\end{align*}
Note that $\langle 2,3,5 \rangle$ is the unique stem extension, where the base normal subgroup is cyclic group $\Z_2$ and the quotient group is the alternating group $A_5$.
We also observe that $A_5$ is the smallest non-abelian finite simple group.

The Poincar\'{e} sphere $S^3_P$ embeds into $\R^5$ by the above, and hence does so with a trivial normal bundle by topological reasons. Using the construction in Section \ref{sec:examples}, we can thus produce a Legendrian embedding $S^3_P\times \Lambda^k_0\subset \R^{2(3+k)+1}$ which admits an exact Lagrangian filling diffeomorphic to $S^3_P\times \R^{k+1}$ for any $k \ge 2$.

Now we show that bounds described in Corollary~\ref{maininequalitystableMorsenumberofafilling} and in Theorems~\ref{intrineqabssi} and \ref{intrineqrelsimaslnonz} are stronger than the bound coming from the homological data of the filling.
Part (ii) of Theorem \ref{intrineqabssi} applied to the perfect group $\langle 2,3,5 \rangle$ implies the bound $|\mathcal Q(S^3_P\times S^1)|\geq 6$ for all representatives. In addition, note that $\OP{Morse}(S^3_P)=6$. This follows from the fact that $S^3_P$ admits a Heegaard splitting of genus $2$, see \cite[Section 9D]{KAL}. Moreover, combining Theorem~\ref{intrineqabssi} with \cite[Proposition 2.1]{OTSMNOACM} we get that even $\OP{stableMorse}(S^3_P)\!=\!6$ holds. Lemma~\ref{lem:stablemorse} now shows that we have $\OP{stableMorse}(S^3_P\times D^{k+1})=6$ as well. In conclusion, the bound we get using the method of Ono-Pajitnov actually equals the Morse number of $S^3_P\times D^{k+1}$. Finally, note that Seidel's isomorphism only predicts that $|\mathcal Q(S^3_P\times S^k)|\geq 2$.

\subsubsection{High-dimensional homology spheres}
We again consider the binary icosahedral group $\langle 2,3,5 \rangle$.
Note that the binary icosahedral group is a finitely presented superperfect group, and by the work of Kervaire \cite{SHSATFG}, see Theorem~\ref{kerealorh1h2vansupperf}, there exists an $n$-dimensional smooth homology sphere that we call $M_{\langle 2,3,5 \rangle}$, where $n\geq 5$, with the property that $\pi_1(M_{\langle 2,3,5 \rangle})\simeq \langle 2,3,5 \rangle$. Using the same arguments as we described in Section~\ref{Poinchomsphest}, we produce a Legendrian manifold $M_{\langle 2,3,5 \rangle}\times S^k$ for $k \gg 1$ sufficiently large, for which the bound
$$|\mathcal Q(M_{\langle 2,3,5 \rangle}\times S^k)|\geq 6$$
thus is satisfied for all representatives. Finally, observe that Seidel's isomorphism again predicts that $|\mathcal Q(M_{\langle 2,3,5 \rangle}\times S^k)|\geq 2$.

\subsection{Exact Lagrangian fillings with finite solvable\\ fundamental group}\label{solvableexamplesfarfromseidelsiso}

In this section, we produce examples for which the estimates described in Section~\ref{estimPajOnogensection} are much stronger compared to the estimate coming from Seidel's isomorphism. In order to do this, we first provide examples of finite solvable groups $G$ with the property that $d(G)-d(G/[G,G]) \gg 0$ is arbitrarily large.

We start with the following observation
\begin{align}\label{absemiddirprreltohsh}
(H\rtimes_{\varphi}K)/[H\rtimes_{\varphi}K, H\rtimes_{\varphi}K]\simeq (K/[K,K])\times ((H/[H,H])/G_H),
\end{align}
where $G_H$ is a subgroup of $H/[H,H]$ generated by $[h]-[\phi(h)]$. The proof can, for example, be deduced from \cite[Proposition 29]{TLCADSOTBGOTSATPS}.
Then we describe a series of finite solvable groups $G_m$ with the property that $d(G_m)=m+1$, $d(G_m/[G_m,G_m])=1$.

\subsubsection*{Example}

Let $\F_q$ be a finite field with $q$ elements, where $q>2$ is a prime number. We define a group $G_m:=\F^m_q \rtimes_{\varphi} \F^{\ast}_q$, where $\varphi: \F^{\ast}_q\to \Aut(\F^m_q)$ acts on the additive group $\F^m_q$ by $\varphi(f)((f_1,\dots,f_m)):=(ff_1,\dots,ff_m)$, $f\in \F^{\ast}_q$, $f_i\in \F_q$. It is easy to see that $G_m$ is a  solvable group. This follows from the existence of the following subnormal series
$$\{1\}< \F_q < \dots < \F^{m-1}_q < \F^m_q < G_m,$$
where $G_m/\F^m_q\simeq \F^{\ast}_q$ and $\F^i_q/\F^{i-1}_q\simeq \F_q$.

Then we observe that from Formula~\ref{absemiddirprreltohsh} it follows that $G_m/[G_m,G_m]\simeq \F^{\ast}_q$ is a non-trivial cyclic group. Therefore, $d(G_m/[G_m,G_m])=1$.

Finally we prove that $d(G_m)=m+1$. Since $d(\F^m_q)=m$ and $d(\F^{\ast}_q)=1$, we get that $d(G_m)\leq m+1$.

We first show that $d(G_m)\geq m$. Assume that $G_m$ has a generating set $S_{G_m}=\{(v_i,g_i)\ |\ 1\leq i\leq k \}$ with $v_i\in\F^m_q$, $g_i\in \F^{\ast}_q$ and $k<m$. Then, note  that every element in the group generated by $S_{G_m}$ has a form $(\sum_{1\leq i\leq k}a_iv_i,h)$, where $h \in \F^{\ast}_q$, $a_i \in \F_q$, and $v_i\in \F^m_q$. This leads to the contradiction with the fact that $d(\F^m_q)=m$.

Then we take a set of generators $S_{G_m}$ and $(v,g)\in S_{G_m}$ with the property that $g$ is a generator of $\F^{\ast}_q$ ($\F^{\ast}_q$ is a cyclic group). Such an element definitely exists since if all the elements of $S_{G_m}$ are of the form $(w,h)$, where $h$ is not a generator of $\F^{\ast}_q$, then $S_{G_m}$ is not a generating set of $G_m$. Again, $\F^{\ast}_q$ is a cyclic group of order $q-1$, and the order of $g$ that we denote by $\OP{ord}(g)$ is coprime to $|\F_q|$. This implies that $\OP{ord}((v,g))=\OP{ord}(g)$ is coprime to $|\F_q|$, and hence none of the primes which divide $\OP{ord}((v,g))$ will divide $[G_m:\langle (v,g) \rangle]=|\F^m_q|$. Let $\pi$ be the set of primes which divide $\OP{ord}(g)$. Then $\langle (v,g) \rangle$ and $\langle (0,g) \rangle$ are two Hall $\pi$-subgroups, and hence by Theorem~\ref{conjpihallsubgr} they are conjugate by some element $x\in G_m$. This implies that $S_{G_m}$, after conjugation, contains an element $x(v,g)x^{-1}=(0,\tilde{g})$, where $\tilde{g}$ is a generator of $\F^{\ast}_q$. We also would like to mention that it is possible to find $x$ explicitly without relying on the theory of Hall $\pi$-subgroups. Then, already knowing that $d(G_m)\geq m$, we can apply the previous argument and see that, in fact, $d(G_m)\geq m+1$. Together with the fact that $d(G_m)\leq m+1$ we get that $d(G_m)=m+1$.

Using Proposition \ref{constrrealizoffundgroup}, we construct an exact Lagrangian filling $M_{G_m}\times L^{k+1}_0$ of a Legendrian $M_{G_m}\times \Lambda^k_0$ inside the standard contact vector space.
Then Theorem \ref{intrineqabssi} tells us that
\begin{align*}
|\mathcal Q(\Lambda)|&\ge  m + 1 -\rank H_1(M_{G_m}\times \R^{k+1};\F) +\sum_{i} \rank H_i(M_{G_m}\times \R^{k+1};\F) \\
& = m +\sum_{i} \rank H_i(M_{G_m}\times \R^{k+1};\F)
\end{align*}
is satisfied for all representatives.

On the other hand, the bound given by Seidel's isomorphism is
$$|\mathcal Q(\Lambda)|\geq \sum_{i} \rank H_i(M_{G_m}\times \R^{k+1};\F).$$
Finally, note that the difference between the previous two bounds gets arbitrarily large as $m\to \infty$.

\clearpage

\address{Department of Mathematics, Uppsala University\\
Box 480, SE-751 06 Uppsala, Sweden\\
\email{dimitroglou.georgios@gmail.com}}

\address{Faculty of Mathematics and Physics, Charles University\\ 
Sokolovsk\'{a} 83, 18000 Praha 8, Czech Republic\\ 
\email{golovko@karlin.mff.cuni.cz}\\
%\received{November 1, 2015}\\
%\accepted{June 3, 2016}
}

\end{document}